\documentclass{amsart}

\usepackage{latexsym}
\usepackage{verbatim}

\usepackage{amssymb,amsmath,amscd,graphicx,color,enumerate}

 \usepackage[normalem]{ulem}
\usepackage[font=bf]{caption}

\usepackage{microtype}

 \usepackage{xcolor}                                           %
\newcommand\op[1]{\psi^{#1}(M; E)}
\newcommand{\clop}{\overline{\psi^{0}}(M; E)}
\newcommand{\clopn}{\overline{\psi^{-1}}(M; E)}
\newcommand{\Hom}{\operatorname{Hom}}
\newcommand{\End}{\operatorname{End}}
\newcommand{\ind}{\operatorname{ind}}
\newcommand{\CI}{\mathcal{C}^{\infty}}

\newcommand{\Ind}{\operatorname{Ind}}
\newcommand{\Prim}{\operatorname{Prim}}
\newcommand{\supp}{\operatorname{supp}}
\newcommand{\coker}{\operatorname{coker}}

\newcommand{\mfkA}{\mathfrak A}
\newcommand{\mfkJ}{\mathfrak J}

\newcommand{\CC}{\mathbb C}

\newcommand{\NN}{\mathbb N}

\newcommand{\RR}{\mathbb R}

\newcommand{\ZZ}{\mathbb Z}

\newcommand{\maB}{\mathcal B}
\newcommand{\maC}{\mathcal C}

\newcommand{\maF}{\mathcal F}

\newcommand{\maH}{\mathcal H}
\newcommand{\maK}{\mathcal K}
\newcommand{\maL}{\mathcal L}

\newcommand{\maQ}{\mathcal Q}
\newcommand{\maR}{\mathcal R}

\newcommand{\maO}{\mathcal O}

\newcommand\pa{\partial}

\newcommand\ede{\, := \,}
\newcommand\seq{\, = \,}

\newtheorem{theorem}{Theorem}[section]
\newtheorem{lemma}[theorem]{Lemma}
\newtheorem{proposition}[theorem]{Proposition}
\newtheorem{corollary}[theorem]{Corollary}

\theoremstyle{definition}
\newtheorem{definition}[theorem]{Definition}
\newtheorem{remark}[theorem]{Remark}
\newtheorem{example}[theorem]{Example}

\newcommand{\Sp}{\mathrm{Sp}}

\begin{document}

\title[Fredholm conditions]{Fredholm conditions and index for
  restrictions of invariant pseudodifferential operators to isotypical
  components}
  
\author[A. Baldare]{Alexandre Baldare}
\email{alexandre.baldare@math.uni-hannover.de} \address{Institut fur
  Analysis, Welfengarten 1, 30167 Hannover, Germany}

\author[R. C\^ome]{R\'emi C\^ome} \email{remi.come@univ-lorraine.fr}
\address{Universit\'{e} Lorraine, 57000 Metz, France}

\author[M. Lesch]{Matthias Lesch} \address{Mathematisches Institut,
  Universit\"at Bonn, Endenicher Allee 60, 53115 Bonn, Germany}
\email{ml@matthiaslesch.de, lesch@math.uni-bonn.de}
\urladdr{www.matthiaslesch.de, www.math.uni-bonn.de/people/lesch}

\author[V. Nistor]{Victor Nistor} \email{nistor@univ-lorraine.fr}
\address{Universit\'{e} Lorraine, 57000 Metz, France}
  \urladdr{http://www.iecl.univ-lorraine.fr/~Victor.Nistor}

  \thanks{
    A.B., R.C., and V.N. have been partially supported by
  ANR-14-CE25-0012-01 (SINGSTAR). V.N. was also supported by the NSF
  grant DMS 1839515. M.L. was partially supported by the Hausdorff Center for
  Mathematics, Bonn.}
%
%

\begin{abstract}
Let $\Gamma$ be a compact group acting on a smooth, compact manifold
$M$, let $P \in \psi^m(M; E_0, E_1)$ be a $\Gamma$-invariant,
classical pseudodifferential operator acting between sections of two
equivariant vector bundles $E_i \to M$, $i = 0,1$, and let $\alpha$ be
an irreducible representation of the group $\Gamma$. Then $P$ induces
a map $\pi_\alpha(P) : H^s(M; E_0)_\alpha \to H^{s-m}(M; E_1)_\alpha$
between the $\alpha$-isotypical components of the corresponding
Sobolev spaces of sections. When $\Gamma$ is finite, we explicitly
characterize the operators $P$ for which the map $\pi_\alpha(P)$ is
Fredholm in terms of the principal symbol of $P$ and the action of
$\Gamma$ on the vector bundles $E_i$. When $\Gamma = \{1\}$, that is,
when there is no group, our result extends the classical
characterization of Fredholm (pseudo)differential operators on compact
manifolds. The proof is based on a careful study of the symbol
$C^*$-algebra and of the topology of its primitive ideal spectrum. We
also obtain several results on the structure of the norm closure of
the algebra of invariant pseudodifferential operators and their
relation to induced representations. Whenever our results also hold
for non-discrete groups, we prove them in this greater generality. As
an illustration of the generality of our results, we provide some
applications to Hodge theory and to index theory of singular quotient
spaces.
\end{abstract}

\maketitle \tableofcontents

\section{Introduction}

Fredholm operators have been extensively studied and appear in many
questions in Mathematical Physics, in Partial Differential Equations
(linear and non-linear), in Algebraic and Differential Geometry, in
Index Theory, and in other areas. On a compact manifold, a classical
pseudodifferential operator is Fredholm between suitable Sobolev
spaces if, and only if, it is elliptic. In this paper, we obtain an
analogous result for the restrictions to isotypical components of a
classical pseudodifferential operator $P$ invariant with respect to
the action of a finite group $\Gamma$ using $C^*$-algebra
methods. Namely, {\em the restriction of $P$ to the isotypical
  component corresponding to an irreducible representation $\alpha$ of
  $\Gamma$ is Fredholm if, and only if, the operator is
  $\alpha$-elliptic (Definition \ref{def.chi.ps} and Theorem
  \ref{thm.main1}).}

Let us now formulate and explain this result in more detail.

\subsection{The setting and general notation}
We shall work essentially in the same setting as the one considered in
\cite{BCLN1}, but for a general finite group $\Gamma$. Thus, {\em
  throughout this paper, $\Gamma$ will be a finite group acting by
  diffeomorphisms on a smooth Riemannian manifold $M$.} As our main
result is only valid for a compact manifold, we assume for this
introduction that $M$ is compact. For the main result (Theorem
\ref{thm.main1}), we do need $\Gamma$ to be discrete and finite, so
our main result is optimal. Again, see \cite{BCLN1}. There is no loss
of generality to assume that $M$ is endowed with an invariant
Riemannian metric, so we will assume that this is the case also.

As usual, $\widehat{\Gamma}$ denotes the finite set of equivalence
classes of irreducible $\Gamma$-modules (or representations). Let $T :
V_0 \to V_1$ be a $\Gamma$-equivariant linear map of $\Gamma$-modules
and $\alpha \in \widehat \Gamma$. Then $T$ induces by restriction a
$\Gamma$-equivariant linear map
\begin{equation}\label{eq.restriction}
    \pi_\alpha(T) : V_{0\alpha} \to V_{1\alpha}
\end{equation}
between the $\alpha$-isotypical components of the $\Gamma$-modules
$V_i$, $i = 0, 1$.

We are mostly interested in this restriction morphism $\pi_\alpha$ in
the following case. Let $P \in \psi^m(M; E_0, E_1)$ be a classical,
$\Gamma$-invariant pseudodifferential operator acting between sections
of two $\Gamma$-equivariant vector bundles $E_i \to M$, $i=0, 1$. Then
we obtain the operator
\begin{equation}\label{eq.def.Pchi}
  \pi_\alpha(P) \, : \, H^s(M; E_0)_\alpha \, \to \, H^{s-m}(M;
  E_1)_\alpha\,,
\end{equation}
which acts between the $\alpha$-isotypical components of the
corresponding Sobolev spaces of sections. Our main result concerns
this operator $\pi_\alpha(P)$. For simplicity, we will consider only
classical pseudodifferential operators in this article
\cite{Hormander3, jager, MelroseScattering, TaylorBook81, Taylor2,
  Treves}.

\subsection{The $\alpha$-principal symbol and $\alpha$-ellipticity}
To put our result into the right perspective, recall that a classical,
order $m$, pseudodifferential operator $P$ is called {\em elliptic} if
its principal symbol
\begin{equation}\label{eq.princ.symb}
  \sigma_m(P) \in \maC^\infty(T^*M \smallsetminus \{0\}; \Hom(E_0, E_1))\,,
\end{equation}
is invertible. Also, recall that a linear operator
  $T : X_0 \to X_1$ acting between Banach spaces is {\em Fredholm} if,
  and only if, the vector spaces
\begin{equation*}
    \ker(T) \ede T^{-1}(0)\ \mbox{ and } \ \coker(T) \ede X_1/TX_0
\end{equation*}
are (both) finite dimensional. Since $M$ is compact, a very well known
and widely used result states that $P : H^s(M; E_0) \to H^{s-m}(M;
E_1)$ is Fredholm if, and only if, $P$ is elliptic, see for instance
\cite{Hormander3, MelroseScattering, Schrohe1, Seeley63, Taylor2,
  VasyNbody} and the references therein.  Consequently, if $P$ is
elliptic, then $\pi_\alpha(P)$ is also Fredholm. The converse is not
true, however, in general. 

To state our main result characterizing the Fredholm property of
$\pi_\alpha(P)$ in terms of the ``$\alpha$-principal symbol''
$\sigma_m^\alpha(P)$ of $P$, Theorem \ref{thm.main1}, we shall need to
introduce $\sigma_m^\Gamma(P)$, the ``$\Gamma$-equivariant principal
symbol'' of $P$, which is a refinement of the principal symbol
$\sigma_m(P)$ of $P$ that takes into account the action of the group
$\Gamma$. The {\em $\alpha$-principal symbol} $\sigma_m^\alpha(P)$ of
$P$ is a suitable restriction of the $\Gamma$-equivariant principal
symbol $\sigma_m^\Gamma(P)$. Let us formulate now the precise
definition of these concepts.

The {\em main question} that we answer in this paper is to determine
when the induced operator $\pi_\alpha(P)$ of Equation
\eqref{eq.def.Pchi} is Fredholm in terms of its $\Gamma$-equivariant
principal symbol $\sigma_m^\Gamma(P)$, see Theorem \ref{thm.main1}
below for the precise statement.

The $\Gamma$-invariance of $P$ implies that its principal symbol is
also $\Gamma$ invariant:
\begin{equation*}
  \sigma_m(P) \in \maC^{\infty}(T^*M \smallsetminus
  \{0\}; \Hom(E_0, E_1))^\Gamma \,.
\end{equation*}
Let $\Gamma_\xi := \{ \gamma \in \Gamma
\, \vert \ \gamma \xi = \xi \}$ denote the isotropy of a $\xi \in
T_x^*M$, $x \in M$, as usual. The isotropy $\Gamma_x$ of $x \in M$ is
defined similarly. Then $\Gamma_\xi \subset \Gamma_x$ acts on $E_{0x}$
and on $E_{1x}$, the fibers of $E_0, E_1 \to M$ at $x$. If $Q \in
\maC^\infty(T^*M \smallsetminus \{0\}; \Hom(E_0, E_1))^\Gamma$, then
$Q(\xi) \in \Hom(E_{0x}, E_{1x})^{\Gamma_\xi}$. Let $\rho \in
\widehat{\Gamma}_\xi$ be an irreducible representation of
$\Gamma_\xi$, then
\begin{equation}\label{eq.hatQ}
  \widehat{Q} (\xi, \rho) \ede \pi_{\rho}\big [Q(\xi) \big] \in
  \Hom(E_{0x \rho}, E_{1x \rho})^{\Gamma_\xi}
\end{equation}
denotes the restriction of $Q$ to the isotypical component
corresponding to $\rho$, with $\pi_\rho$ defined in Equation
\eqref{eq.restriction}. Let
\begin{equation}\label{eq.def.XMG}
  X_{M, \Gamma} \ede \{ ( \xi, \rho) \, \vert \ \xi \in T^*M
  \smallsetminus \{0\} \mbox{ and } \rho \in \widehat \Gamma_\xi \}
  \,.
\end{equation}
Thus $Q$ defines a function on $X_{M, \Gamma}$. Applying this
construction to $\sigma_m(P) \in \maC^\infty(T^*M \smallsetminus
\{0\}; \Hom(E_0, E_1))^\Gamma$ we obtain a function, the {\em
    $\Gamma$-principal symbol}
\begin{equation}\label{eq.def.Gamma.symb}
  \begin{gathered}
  \sigma_m^\Gamma(P) : X_{M, \Gamma} \to \bigcup_{(x, \rho) \in X_{M,
      \Gamma}} \Hom(E_{0x \rho}, E_{1x \rho})^{\Gamma_\xi}\,, \\
  \sigma_m^\Gamma (P) (\xi, \rho) \ede \pi_{\rho}(\sigma_m(P)(\xi))
  \in \Hom(E_{0x \rho}, E_{1x \rho})^{\Gamma_\xi}\,, \ \ \xi \in
  T_x^*M\,.
   \end{gathered}
\end{equation}
That is $\sigma_m^\Gamma (P) := \widehat{\sigma_m (P)}$.

The $\alpha$-principal symbol $\sigma_m^\alpha(P)$ of $P$, $\alpha \in
\widehat{\Gamma}$, is defined in terms of $\sigma_m^\Gamma(P)$, but we
need a crucial additional ingredient that takes $\alpha$ into account.

Recall that $\Gamma_{g\xi}=g\Gamma_\xi g^{-1}$ and that this defines
an action of $\Gamma$ on the set of stabilizer subgroups
$\operatorname{Stab}_\Gamma(T^*M) := \{\Gamma_\xi \mid \xi \in T^*M\}$
given by $g\cdot \Gamma_\xi=\Gamma_{g\xi}$. For $\rho \in
\widehat{\Gamma}_\xi$ define $g \cdot \rho \in
\widehat{\Gamma}_{g\xi}$ by $(g\cdot \rho)(h)=\rho(g^{-1}hg)$, for all
$h\in \Gamma_{g\xi }$. Let $\Gamma_0 \subset \Gamma$ be a
  minimal element (for inclusion) among the isotropy groups $\Gamma_x$
  of elements $x \in M$. Such a minimal element exists trivially,
  since $\Gamma$ is finite. Moreover, if $M/\Gamma$ is connected,
  then $\Gamma_0$ is unique up to conjugacy (see Subsection
  \ref{ssec.principal}). We assume for a moment that this is the case,
  that is, that $M/\Gamma$ is connected. Then we let
\begin{equation}\label{eq.def.Xalpha}
   X^{\alpha}_{M, \Gamma} \ede \{(\zeta, \rho) \in X_{M, \Gamma}\,
   \mid\ \exists g\in \Gamma,\ \Hom_{\Gamma_0}(g\cdot \rho, \alpha)
   \neq 0\, \, \}\,.
\end{equation}
(Note that it is implicit in the definition of $X^{\alpha}_{M,
  \Gamma}$ that $\Gamma_0 \subset \Gamma_{g \zeta} = g \cdot
\Gamma_{\zeta}$.) In general (if $M/\Gamma$ is not connected),
  we define $X^\alpha_{M, \Gamma}$ by taking the disjoint union of the
  corresponding spaces for each connected component of $M/\Gamma$, see
  Subsection \ref{ssec.disconnected}.

\begin{definition}\label{def.chi.ps}
The {\em $\alpha$-principal symbol} $\sigma_m^\alpha (P)$ of $P$ is
the restriction of the \emph{$\Gamma$-principal symbol}
$\sigma_m^\Gamma(P)$ to $X^\alpha_{M, \Gamma}$:
\begin{equation*}
  \sigma_m^\alpha(P) \ede \sigma_m^\Gamma(P)\vert_{X^\alpha_{M,
      \Gamma}}\,.
\end{equation*}
We shall say that $P \in \psi^m(M; E_0, E_1)^{\Gamma}$ is {\em
  $\alpha$-elliptic} if its $\alpha$-principal symbol
$\sigma_m^\alpha(P)$ is invertible everywhere on its domain of
definition. Note that when $(\xi,\rho) \in X_{M,\Gamma}^\alpha$ is
such that $E_{x\rho} = 0$, then $\sigma_m^\Gamma(P)(\xi,\rho) : 0 \to
0$ is always invertible.
\end{definition}

\subsection{Statement of the main result}
An alternative formulation of Definition \ref{def.chi.ps} is that $P$
is $\alpha$-elliptic if, and only if, $\sigma_m^\Gamma$ is invertible
on $X_{M, \Gamma}^\alpha$ (this is, of course, a condition only for
those $\rho$ such that $E_{i \rho} \neq 0$, because, otherwise, we get
an operator acting on the zero spaces, which we admit to be
invertible). We then have the following result extending the classical
result (i.e.\ $\Gamma = \{1\}$) and the one from \cite{BCLN1}
(i.e.\ $\Gamma$ finite abelian) to a general finite group $\Gamma$.

\begin{theorem}\label{thm.main1}
Let $\Gamma$ be a finite group acting on a smooth, compact manifold
$M$ and let $P \in \psi^m(M; E_0, E_1)^{\Gamma}$ be a
$\Gamma$-invariant classical pseudodifferential operator acting
between sections of two $\Gamma$-equivariant bundles $E_i \to M$, $i =
0, 1$, $m \in \RR$, and $\alpha \in \widehat{\Gamma}$. We have that
\begin{equation*}
    \pi_\alpha(P) : H^s(M; E_0)_\alpha \, \to \, H^{s-m}(M;
    E_1)_\alpha
\end{equation*}
is Fredholm if, and only if it is $\alpha$-elliptic.
\end{theorem}

As in the abelian case, if $\Gamma$ acts without fixed points on a
{\em dense} open subset of $M$, then $X_{M, \Gamma} = X^\alpha_{M,
  \Gamma}$ for all $\alpha \in \widehat{\Gamma}$, by Corollary
\ref{cor.all}. Hence, in this case, $P$ is $\alpha$-elliptic if, and
only if, it is elliptic. The ellipticity of $P$ can thus be checked in
this case simply by looking at the action of $P$ on a single
isotypical component. We stress, however, that if $\Gamma$ is not
discrete, this statement, as well as the statement of the above
theorem, are no longer true. However, many intermediate results remain
valid for compact Lie groups.

A motivation for our result comes from index theory. Let us assume
that $P$ is $\Gamma$-invariant and elliptic. Atiyah and Singer have
determined, for any $\gamma \in \Gamma$, the value at $\gamma$ of the
character of $\ind_\Gamma(P) \in R(G)$. More precisely, they have
computed $\ind_\Gamma(P)(\gamma) \in \CC$ in terms of data at the
fixed points of $\gamma$ on $M$ \cite{AS3}. (Here $R(G) :=
\ZZ^{\widehat{G}}$ is the representation ring of $G$ and is identified
with a subalgebra of $\CI(G)^G$, the ring of conjugacy invariant
functions on $G$ via the characters of representations.)  By contrast,
the multiplicity of $\alpha \in \widehat{\Gamma}$ in $\ind_\Gamma(P)$
was much less studied. It did appear, however, implicitly in the work
of Br\"uning \cite{Bruning78}, who initiated the program of studying
the ``isotypical heat trace'' $\mathrm{tr}( p_\alpha e^{-t\Delta} )$
and its short time asymptotic expansion. Its heat trace is nothing but
the heat trace of $\pi_\alpha(\Delta)$. This question was addressed
also by Paradan and Vergne, who obtained several important related
results, see \cite{PVActa} and the references therein. Br\"uning's
program would lead, in particular, to a heat equation determination of
the $\alpha$-isotypical component of the $\Gamma$-equivariant index
$\ind_\Gamma(D)$ for Dirac type operators $D$. This program is one of
the motivations of this paper.

Indeed, we obtain that the (Fredholm) index of $\pi_{\alpha}(P)$
depends only on the homotopy class of its $\alpha$-principal
symbol. Some results are contained, however, in Theorem
\ref{thm.index} and in the remark following it. In particular, this
yields results on the index theory of singular quotient spaces. We
therefore expect our results to have applications to the Hodge theory
of algebraic varieties \cite{AlbinHodge, MazzeoHodge, Brion,
  CheegerHodge, GriffithHarrisBook}, see Remark \ref{rem.Hodge}. In
the case of a non discrete compact Lie group, the computation of this
index is related to the index class of $G$-transversally elliptic
operators initiated in \cite{atiyahGelliptic,singer73}. Since then,
this has been studied in $K$-theory
\cite{Baldare:KK,Benameur:Baldare,Julg82,Kasparov2016} and in
equivariant cohomology \cite{Baldare:H, BV1,BV2,PVTrans}.
 
Our proof makes heavy use of $C^*$-algebras by considering the natural
$C^*$-algebra completions of algebras of pseudodifferential operators
and related $C^*$-algebras. This point of view has been used and
advocated of course by many people. $C^*$-algebras were used very
recently to obtain Fredholm conditions in \cite{dMG,LMN, MPR}, for
example.  Some of the algebras involved were groupoid algebras
\cite{AnSk11a, AnSk11b, CNQ, DS2, Mo1, Re}. Fredholm conditions play
an important role in the study of the essential spectrum of Quantum
Hamiltonians \cite{BLLS1, GI, Ge, HM, LS}. The technique of ``limit
operators'' \cite{LR, Li, LiS, RRS} is related to groupoids. Some of
the most recent papers using related ideas include \cite{Zhang, CCQ,
  CNQ, Remi, MougelH, MaNi, vEY2}, to which we refer for further
references. Besides $C^*$-algebras, pseudodifferential operators were
also used to obtain Fredholm conditions, see, for example, \cite{DLR,
  LauterMoroianu1, LeschVertman, MelroseScattering,
  MazzeoMelroseAsian, Schrohe1, Schrohe2, SchulzeBook98} and the
references therein. In addition to the works already mentioned,
several general results on $C^*$ and related algebras related to this
work were obtained by Cordes and McOwen \cite{CordesCstar}, Melo,
Nest, and Schrohe \cite{SchroheCstar}, Melrose and Nistor
\cite{MelroseNistorK}, Rabinovich, Schulze, and Tarkhanov \cite{Rab3},
Taylor \cite{TaylorG}, Voiculescu \cite{Voiculescu1, Voiculescu2,
  Voiculescu3}, and others. See \cite{CalderonZygmund1,
  CalderonZygmund2, CordesAlg, KohnNirenbergAlg, gokhberg52,
  gokhberg60, HormanderAlg, mihlin48, mihlin56} for some older,
related results on singular integral operators.

\subsection{Contents of the paper} 
\label{sub:Contents of the paper}
We start in Section \ref{sec.Preliminaries} with some
preliminaries. We recall some facts about group actions, most notably
the induction of representations and Frobenius reciprocity for finite
groups. We also review some notions concerning the primitive spectrum
of a $C^*$-algebra, as well as basic facts concerning (equivariant)
pseudodifferential operators.

As in \cite{BCLN1}, we may assume $E_0 = E_1 = E$ and $P$ to be of
order zero. Let $A_M := \maC(S^*M;\End(E))$. The most substantial
technical results are in Section \ref{sec.principal.symbol}. There, we
introduce the subset $\Omega_M:=\{(\xi,\rho)\in S^*M\times
\hat{\Gamma}_\xi,\ \rho \subset E_\xi\}$ of $X_{M,\Gamma}$ described
above and identify the primitive spectrum of the $C^*$-algebra
$A_M^\Gamma$ of $\Gamma$-invariant symbols with the set
$\Omega_M/\Gamma$.  Some care is taken to describe the corresponding
topology on $\Omega_M/\Gamma$. We then consider the canonical map from
$A_M^\Gamma$ to the Calkin algebra of $L^2(M;E)_\alpha$ and show that
the closed subset of $\Prim (A_M^\Gamma)$ associated to its kernel is
$X_{M,\Gamma}^\alpha/\Gamma$.

These descriptions are used in Section \ref{sec.applications} to prove
the main result of the paper, Theorem \ref{thm.main1}. This section
also addresses some particular cases of the Theorem and gives a few
examples. We also explain the relation with previously known results,
namely:
  the particular formulation in the abelian case, which was
    established in \cite{BCLN1};
  Fredholm conditions for transversally elliptic operators when
    the group $\Gamma$ is not discrete; and
  Simonenko's local principle for Fredholm operators.

We thank Claire Debord, Paul-Emile Paradan, Elmar Schrohe, Georges
Skandalis, and Andrei Teleman for useful discussions. The last named
author thanks Max Planck Institute for support while this research was
performed.

\section{Preliminaries}
\label{sec.Preliminaries}

This section is devoted to background material. For the most part, it
will consist of a brief review of sections 2 and 3 of \cite{BCLN1},
where the reader will find more details, as well as definitions and
results not repeated here. Note, however, that we need certain
preliminary results for the case $\Gamma$ non-commutative that were
not needed in the abelian case. Nevertheless, the reader familiar with
\cite{BCLN1} can skip this section at a first reading.

For simplicity, let us {\em from now on for the rest of this paper
  also assume that $M/\Gamma$ is connected, except in Subsection
    \ref{ssec.disconnected},} where we explain how the disconnected
  case reduces to the connected case.

\subsection{Group representations}
\label{ssec.group.r}
We follow the standard terminology and conventions. See, for instance,
\cite{BCLN1, tomDieckRepBook, SerreBook}, where one can find further
details. Most of the needed basic background material was recalled in
greater detail in \cite{BCLN1}.

Throughout the paper, we denote by $\Gamma$ a \emph{finite} group
acting by isometries on a smooth, Riemannian manifold $M$ (without
boundary). We use the standard notations, see \cite{BCLN1,
  tomDieckRepBook, SerreBook}, to which we refer for further details.
If $x \in M$, then $\Gamma x$ is the $\Gamma$ orbit of $x$ and
\begin{equation}
   \Gamma_x := \{ \gamma \in \Gamma \, \vert \ \gamma x = x \} \subset
   \Gamma
\end{equation}
the isotropy group of the action at $x$.

We shall write $H \sim H'$ if the subgroups $H$ and $H'$ are
conjugated in $\Gamma$. If $H \subset \Gamma$ is a subgroup, then
$M_{(H)}$ will denote the set of elements of $M$ whose isotropy
$\Gamma_{x}$ is conjugated to $H$ (in $\Gamma$), that is, the set of
elements $x \in M$ such that $\Gamma_x \sim H$.

Assuming that $\Gamma$ acts on a space $X$, we denote by $\Gamma
\times_{H} X$ the space
\begin{equation}\label{eq.associate}
\Gamma \times_H X  \ede (\Gamma \times  X)/\sim, 
\end{equation}
where $(\gamma h,x)\sim (\gamma ,hx)$, $\forall \gamma \in \Gamma,
h\in H$ and $x\in X$.  

Let $V$ be a normed complex vector space and $\maL(V)$ be the set of
bounded operators on $V$. A representation of $\Gamma$ on $V$ is a
group morphism $\Gamma \to \maL(V)$; in that case we also call $V$ a
$\Gamma$-module.

For any two $\Gamma$-modules $\maH$ and $\maH_1$, we shall denote by
\begin{equation*}
    \Hom_{\Gamma}(\maH, \maH_1) \seq \Hom(\maH, \maH_1)^\Gamma \seq
    \maL(\maH, \maH_1)^\Gamma
\end{equation*}
the set of continuous linear maps $T : \maH \to \maH_1$ that commute
with the action of $\Gamma$, that is, $T (\gamma \xi) = \gamma T(\xi)$
for all $\xi \in \maH$ and $\gamma \in \Gamma$.

Let $\maH$ be a $\Gamma$-module and $\alpha$ an irreducible
$\Gamma$-module. Then $p_\alpha$ will denote the $\Gamma$-invariant
projection onto the $\alpha$-isotypical component $\maH_\alpha$ of
$\maH$, defined as the largest (closed) $\Gamma$ submodule of $\maH$
that is isomorphic to a multiple of $\alpha$. In other words,
$\maH_\alpha$ is the sum of all $\Gamma$-submodules of $\maH$ that are
isomorphic to $\alpha$. Notice that $\maH_{\alpha}
\simeq \alpha \otimes \Hom_{\Gamma}(\alpha, \maH)$.

Since $\Gamma$ is finite, it is, in particular, compact, and hence we
have
\begin{equation}\label{eq.isotypical}
  \maH_\alpha \, \neq \, 0 \ \Leftrightarrow \ \Hom_\Gamma(\alpha,
  \maH) \, \neq \, 0 \ \Leftrightarrow \ \Hom_\Gamma(\maH, \alpha) \,
  \neq \, 0 .
\end{equation}
If $T \in \maL(\maH)^\Gamma$ (i.e.\ $T$ is $\Gamma$-equivariant), then
$T(\maH_\alpha) \subset \maH_\alpha$ and we let
\begin{equation}\label{eq.restriction2}
    \pi_\alpha : \maL(\maH)^\Gamma \to \maL(\maH_\alpha) \,, \quad
    \pi_\alpha(T) \ede T\vert_{\maH_\alpha}\,,
\end{equation}
be the associated morphism, as in Equation \eqref{eq.restriction} of
the Introduction. The morphism $\pi_\alpha$ will play an essential
role in what follows.

\subsection{Induction and Frobenius reciprocity}
\label{ssec.Frobenius}
We now recall some definitions and results for induced representations
mainly to set up notation and to obtain some intermediate results. 

We let $V^{(I)} := \{ f : I \to V \}$ for $I$ finite. If $H \subset
\Gamma$ is a subgroup (hence also finite) and $V$ is a $H$-module, we
define, as usual,
\begin{equation}\label{eq.def.induced}
  \begin{split}
    \Ind_H^\Gamma (V)  \ede \ & \CC[\Gamma] \otimes_{\CC[H]} V \\ 
   \simeq \ &
  \{\, f : \Gamma \to V \, |\ f(gh^{-1}) = h f(g) \, \} \,
  \simeq \, V^{(\Gamma/H)}\,
  \end{split}
\end{equation}
to be the {\em induced representation}.
The last isomorphism is obtained using a set of representatives of
the right cosets $\Gamma/H$. The action of the group $\Gamma$ on
$\Ind_H^\Gamma (V)$ is obtained from the left multiplication on
$\CC[\Gamma]$ and the first isomorphism defines the $\Gamma$-module
structure on $\Ind_H^\Gamma (V)$. The induction is a functor, that is,
the $\Gamma$-module $\Ind_H^\Gamma (V)$ depends functorially on $V$.

\begin{remark}\label{rem.CD} 
Summarizing Remark 2.2 of \cite{BCLN1}, we have that
\begin{enumerate}
\item if $V$ is a $H$-algebra, then $\Ind_H^\Gamma(V)$ is an algebra
  for the pointwise product,
 \item if $V$ is a left $R$-module (with compatible actions of
   $\Gamma$), then $\Ind_H^\Gamma (V)$ is a $\Ind_H^\Gamma (R)$
   module, again with the pointwise multiplication,
 \item the induction is compatible with morphisms of modules and
   algebras (change of scalars), again by the function representation
   of the induced representation.
\end{enumerate}
See \cite[Remark 2.2]{BCLN1} for more details.
\end{remark}

We shall use the {\em Frobenius reciprocity} in the form that states
that we have an isomorphism
\begin{equation}\label{eq.Frobenius}
  \begin{gathered}
    \Phi = \Phi_{H, V}^{\Gamma, \maH} : \Hom_{H}( \maH, V ) \, \to \,
    \Hom_{\Gamma}( \maH, \Ind_H^\Gamma(V) ) \,,\\ \Phi(f)(\xi) \ede
    \frac{1}{|H|} \, \sum_{g \in \Gamma} g \otimes_{\CC[H]} f(g^{-1}
    \xi),\quad \xi \in \maH,\ f \in \Hom_{H}( \maH, V ) \,.
  \end{gathered}
\end{equation}
The version of the Frobenius reciprocity used in this paper is valid
only for finite groups \cite{tomDieckRepBook, SerreBook} (although it
can be suitably be generalized to the compact case). We note that a
more precise notation would be to write $\Hom_{H}(
\operatorname{Res_H^\Gamma}(\maH), V)$ instead of our simplified
notation $\Hom_{H}(\maH, V)$.

\begin{definition}\label{def.associated}
Let $A$ and $B$ be finite groups and let $H$ a subgroup of both
  $A$ and $B$. Let $\alpha \in \widehat{A}$ and $\beta \in
  \widehat{B}$. We say that $\alpha $ and $\beta$ are {\em
    $H$-disjoint} if $\mathrm{Hom}_H(\alpha ,\beta)=0$, otherwise we
  say that they are {\em $H$-associated (to each other)}.
\end{definition}

Let $\alpha \in \widehat{\Gamma}$, let $H \subset \Gamma$ be a
subgroup, and $\beta \in \widehat{H}$. A useful consequence of the
Frobenius reciprocity is that the multiplicity of $\alpha$ in
$\ind_H^\Gamma(\beta)$ is the same as the multiplicity of $\beta$ in
the restriction of $\alpha$ to $H$. In particular, $\alpha$ is
contained in $\ind_H^\Gamma(\beta)$ if, and only if, $\beta$ is
contained in the restriction of $\alpha$ to $H$, in which case recall
that we say that $\alpha$ and $\beta$ are $H$-associated (Definition
\ref{def.associated}). On the other hand, recall that if $\beta$ is
    {\em not} contained in the restriction of $\alpha$ to $H$, we say
    that $\alpha$ and $\beta$ are {\em $H$-disjoint.}

Let $V$ be a $H$-module and $\maH$ be the trivial
  $\Gamma$-module $\CC$. Then we obtain, in particular, an isomorphism
\begin{equation}\label{eq.Frobenius2}
  \begin{gathered}
    \Phi : V^{H} \seq \Hom_{H}( \CC, V ) \, \simeq\, \Hom_{\Gamma}(
    \CC, \Ind_H^\Gamma(V) ) \seq \Ind_H^\Gamma(V)^\Gamma \,,\\
    \Phi(\xi) \ede \frac{1}{|H|} \, \sum_{g \in \Gamma} g
    \otimes_{\CC[H]} \xi \seq \sum_{x \in \Gamma/H} x \otimes \xi\,.
  \end{gathered}
\end{equation}
If $V$ is an algebra, then the map $\Phi$ is an isomorphism of
algebras. In particular, we obtain the following consequences.

\begin{remark}\label{rem.new}
Let $H \subset \Gamma$ be a subgroup of $\Gamma$, $\beta_j$ be
  non-isomorphic simple $H$-modules, $j = 1,\ldots, N$, and
\begin{equation}\label{eq.def.beta}
   \beta \ede \oplus_{j=1}^N
   \beta_j^{k_j} \,.
\end{equation}
We then have that $\Ind_H^\Gamma(\beta) \simeq \oplus_{j=1}^N
\Ind_H^\Gamma(\beta_j^{k_j})$ and the Frobenius isomorphism gives
\begin{equation}\label{eq.direct.sum}
  \Ind_H^\Gamma(\End(\beta))^\Gamma  \simeq \End(\beta)^{H} 
  \simeq \oplus_{j=1}^N \End(\beta_j^{k_j})^H  \simeq
  \oplus_{j=1}^N M_{k_j}(\CC),
\end{equation}
which is a semi-simple algebra and where the first isomorphism is
induced by $\Phi$ of Equation \eqref{eq.Frobenius2}.
\end{remark}

We shall need the following refinement of the above remark.

\begin{lemma}\label{lemma.componentwise}
Let $\beta \ede \oplus_{j=1}^N \beta_j^{k_j}$ be as in Equation
\eqref{eq.def.beta}, let
\begin{equation*}
 T \seq (T_j) \in \End(\beta)^{H}  \, \simeq \, 
   \oplus_{j = 1}^N \End(\beta_j^{k_j})^H\,,
\end{equation*}
with $T_j \in \End(\beta_j^{k_j})^{H}$, and let $\xi_j \in
\Ind_H^\Gamma(\beta_j^{k_j})$. We let
\begin{equation*}
  \xi \ede (\xi_j) \in
  \oplus_{j=1}^N \Ind_H^\Gamma(\beta_j^{k_j}) \simeq
  \Ind_H^\Gamma(\beta) \,.
\end{equation*}
Then $\Phi(T) (\xi) = (\Phi(T_j) \xi_j)_{j=1,\ldots,N}$.
\end{lemma}

\begin{proof}
  See for example \cite[Lemma 2.4]{BCLN1}.
\end{proof}

For the abelian case, the following
elementary result was proved in \cite{BCLN1} Proposition 2.5.
That proof {\em does not} generalize to our case.

\begin{proposition}\label{prop.res.alpha}
Let $\beta \ede \oplus_{j=1}^N \beta_j^{k_j}$ be as in Equation
\eqref{eq.def.beta}. Let $J \subset \{1, 2, \ldots, N\}$ be the set of
indices $j$ such that $\alpha$ and $\beta_j$ are $H$-disjoint
(i.e.\ $\beta_j$ is not contained in the restriction of $\alpha$ to
$H$). Then the morphism
\begin{equation*}
  \pi_\alpha : \Ind_H^\Gamma(\End(\beta))^\Gamma \to
  \End(p_\alpha \Ind_H^\Gamma(\beta))
\end{equation*}
is such that
\begin{equation*}
  \ker(\pi_\alpha) \seq \bigoplus \limits_{j \in J}
  \Ind_H^\Gamma(\End(\beta_j^{k_j}))^\Gamma \ \mathrm{and}
  \ \operatorname{Im}(\pi_\alpha) \, \simeq \, \bigoplus\limits_{j \notin J}
  \Ind_H^\Gamma(\End(\beta_j^{k_j}))^\Gamma \,.
\end{equation*}
\end{proposition}

\begin{proof}
By Lemma \ref{lemma.componentwise}, we can assume that $N = 1$.
Therefore the algebra $\End(\beta)^H$ is simple (more precisely,
isomorphic to a matrix algebra $M_q(\CC)$, $q = k_1$). We shall use
the isomorphism of Equation \eqref{eq.direct.sum}. The action of
$\Ind_H^\Gamma(\End(\beta))^\Gamma \simeq \End(\beta)^H \simeq
M_q(\CC)$ on $\Ind_H^\Gamma(\beta)$ is unital (i.e.\ non-degenerate),
so the morphism
\begin{equation}
  M_q(\CC) \, \simeq \, \Ind_H^\Gamma(\End(\beta))^\Gamma \to
  \End(p_\alpha \Ind_H^\Gamma(\beta))
\end{equation}
is injective if, and only if, $p_\alpha \Ind_H^\Gamma(\beta) \neq
0$. Notice the following equivalences
\begin{equation*}
   p_\alpha \Ind_H^\Gamma(\beta) \neq 0\Leftrightarrow \Hom(\alpha,
   \Ind_H^\Gamma(\beta))^\Gamma \neq 0 \Leftrightarrow\Hom(\alpha,
   \beta)^H \neq 0.
\end{equation*}
The result then follows from Equation \eqref{eq.isotypical}.
\end{proof}

\subsection{The primitive ideal spectrum of a $C^*$-algebra}
\label{ssec.primitive.spec} We shall need a few basic
concepts and facts about $C^{*}$-algebras. A general reference is
\cite{Dixmier}.  Recall that a two-sided ideal $I \subset A$ of a
$C^{*}$-algebra $A$ is called {\em primitive} if it is the kernel of a
non-zero, irreducible $*$-representation of $A$. Hence $A$ is {\em
  not} a primitive ideal of itself. By $\Prim(A)$ we shall denote the
set of primitive ideals of $A$, called the {\em primitive ideal
  spectrum} of $A$. If $X$ is a locally compact space, then
$\maC_0(X)$ denotes the space of continuous functions $X \to \CC$ that
vanish at infinity. The concept of primitive ideal spectrum is
important for us since we have a natural homeomorphism
\begin{equation}\label{eq.Gelfand}
    \Prim(\maC_0(X)) \, \simeq \, X\,.   
\end{equation}
This identification lies at the heart of non-commutative geometry
\cite{ConnesBook}. See also \cite{gayral, ManinBook}.

If $A$ is a type I $C^*$-algebra, then $\Prim(A)$ identifies with the
set of isomorphism classes of irreducible representations of $A$. Any
$C^*$-algebra with only finite dimensional irreducible representations
is a type I algebra \cite{Dixmier}. Most of the algebras considered in
this paper (a notable exception are the algebras of compact
operators), have this property.

The following example from \cite{BCLN1} will be used several times.

\begin{example}\label{ex.group}
Let $H$ be a finite group and $\beta = \oplus_{j=1}^N \beta_j^{k_j}$
be as in Remark \ref{rem.new}. Then, as explained in that remark,
$\maL(\beta)^H \simeq \oplus_{j} M_{k_j}(\CC)$.  The algebra
$\maL(\beta)^H = \End_H(\beta)$ is thus a $C^*$-algebra with only
finite dimensional representations and we have natural homeomorphisms
\begin{equation*}
  \Prim(\End_H(\beta)) \, \leftrightarrow \, \{\beta_1, \beta_2,
  \ldots, \beta_N\} \, \leftrightarrow \, \{1, 2, \ldots, N\}\,.
\end{equation*}
\end{example}

The space $\Prim(A)$ is a topological space for the Jacobson topology:
we refer to \cite{Dixmier} for more details. We will recall some facts
about this topology when we need it, see Lemma \ref{lemma.topology}
below.

We shall need the following ``central character'' map.

\begin{remark}\label{rem.central.char}
Let $Z$ be a commutative $C^*$-algebra and $\phi : Z \to M(A)$ be a
$*$-morphism to the multiplier algebra $M(A)$ of $A$ \cite{APT,
  Busby}. Assume that $\phi(Z)$ commutes with $A$ and $\phi(Z)A =
A$. Then Schur's lemma gives that every irreducible representation of
$A$ restricts to (a multiple of) a character of $Z$ and hence there
exists a natural continuous map
\begin{equation}\label{eq.centr.char}
  \phi^* : \Prim(A) \to \Prim(Z) \,,
\end{equation}
which we shall call also the {\em central character map} (associated
to $\phi$).
\end{remark}

We conclude our discussion with the following simple result.

\begin{lemma}\label{lemma.product}
We freely use the notation of Example \ref{ex.group}. The inclusion of
the unit $\CC \to \End_H(\beta)$ induces a morphism $j : \maC_0(X) \to
\maC_0(X; \End_H(\beta)) \simeq \maC_0(X) \otimes \End_H(\beta)$. The
resulting central character map is the first projection
\begin{equation}\label{eq.first.proj}
   j^* : \Prim(\maC_0(X; \End_H(\beta))) \simeq X \times \{1, 2,
   \ldots, N\} \to X \simeq \Prim(\maC_0(X)) \,.
\end{equation}
\end{lemma}

\subsection{Group actions on manifolds}
\label{ssec.group.actions}
As before, we consider a finite group $\Gamma$ acting by isometries on
a compact Riemannian manifold $M$.

\subsubsection{Slices and tubes}
\label{ssec.slice}
Given $x \in M$, the isotropy group $\Gamma_x$ acts linearly and
isometrically on $T_xM$. For $r > 0$, let $U_x := (T_xM)_r$ denote the
set of vectors of length $< r$ in $T_xM$. It is known then that, for
$r > 0$ small enough, the exponential map gives a $\Gamma$-equivariant
isometric diffeomorphism
\begin{equation}\label{eq.def.tube}
  W_x =\exp(\Gamma \times_{\Gamma_x} U_x)\simeq \Gamma
  \times_{\Gamma_x} U_x
\end{equation}
where $W_x$ is a $\Gamma$-invariant neighborhood of $x$ in $M$ and
$\Gamma \times_{\Gamma_x} U_x$ is defined in equation
\eqref{eq.associate}.  More precisely, $W_x$ is the set of $y \in M$
at distance $<r$ to the orbit $\Gamma x$, if $r > 0$ is small
enough. The set $W_x$ is called a {\em tube} around $x$ (or $\Gamma
x$) and the set $U_x$ is called the {\em slice} at $x$. When $M$ is
compact, the injectivity radius is bounded from below, so we may
assume that the constant $r$ does not depend on $x$.

\subsubsection{Equivariant vector bundles}
Let us consider now a $\Gamma$-equivariant smooth vector bundle $E \to
M$. Let us fix $x \in M$ and consider as above the tube $W_x \simeq
\Gamma \times_{\Gamma_x} U_x$ around $x$, see Equation
\eqref{eq.def.tube}. We use this diffeomorphism to identify $U_x$ to a
subset of $M$, in which case, we can also assume the restriction of
$E$ to the slice $U_x$ to be trivial. Therefore, there exists a
$\Gamma_x$-module $\beta$ such that
\begin{equation}\label{eq.trivial.E}
  \begin{gathered}
    E\vert_{U_x} \, \simeq \, U_x \times \beta \ \mbox{ and
    }\\ E\vert_{W_x} \, \simeq \, \Gamma \times_{\Gamma_x} (U_x \times
    \beta)\,,
  \end{gathered}
\end{equation}
The second isomorphism is $\Gamma$-equivariant.

Assume $E$ is endowed with a $\Gamma$-invariant hermitian metric. We
then have isomorphisms of $\Gamma$-modules:
\begin{equation}\label{eq.induced}
  \begin{gathered}
    L^2(W_x; E\vert_{W_x})  \simeq \Ind^\Gamma_{\Gamma_x} ( L^2(U_x;
    \beta))\  \text{ and} \\
    \maC_0(W_x; E\vert_{W_x})  \simeq
  \Ind^\Gamma_{\Gamma_x} ( \maC_0(U_x; \beta)) \,.
  \end{gathered}
\end{equation}
In view of the previous isomorphism, we will often identify $W_x$ and
$\Gamma \times_{\Gamma_x} U_x$, making no distinction between them to
simplify notations.

\subsubsection{The principal orbit bundle}\label{ssec.principal}
Recall that $M_{(H)}$ denotes the set of points of $M$ whose
stabilizer is conjugated in $\Gamma$ to $H$. Recall that we have
assumed that $M/\Gamma$ is \emph{connected}. It is known then
\cite{tomDieckTransBook} that there exists a {\em minimal isotropy}
subgroup $\Gamma_0 \subset \Gamma$, in the sense that $M_{(\Gamma_0)}$
is a dense open subset of $M$, with measure zero complement in $M$.
 
In particular, the fact that $M/\Gamma$ is connected gives that there
exist minimal elements for the set of isotropy groups of points in $M$
(with respect to inclusion) and all minimal isotropy groups are
conjugated to a fixed subgroup $\Gamma_0 \subset \Gamma$. By the
definition, the set $M_{(\Gamma_0)}$ consists of the points whose
stabilizer is conjugated to that minimal subgroup. The set
$M_{(\Gamma_0)}$ is called the {\em principal orbit bundle} of $M$.
We will denote $M_{(\Gamma_0)}$ by $M_0$ in the sequel.

The principal orbit bundle $M_0 := M_{(\Gamma_0)}$ has the following
useful property. If $x \in M_0$, then $\Gamma_x$ acts {\em trivially}
on the slice $U_x$ at $x$, by the minimality of $\Gamma_0$.  Hence
$\Gamma_0$ acts trivially on $T^*_xM$ as well, which implies that
$\Gamma_0 \subset \Gamma_\xi$ for any $\xi \in T^*_x M$. If, on the
other hand, $x \in M$ is arbitrary (not necessarily in the principal
orbit bundle), then the isotropy of $\Gamma_x$ will contain a subgroup
conjugated to $\Gamma_0$.

\subsection{Pseudodifferential operators}
\label{sub.psdo}
We continue to follow \cite{BCLN1}.  We also continue to assume that
$\Gamma$ is a finite Lie group that acts smoothly and isometrically on
a smooth Riemannian manifold $M$.  Let $\op{m}$ denote the space of
order $m$, {\em classical} pseudodifferential operators on $M$ with
\emph{compactly supported} distribution kernel.

Let $\clop$ and $\clopn$ denote the respective norm closures of $\op{0}$ and
$\op{-1}$. The action of $\Gamma$ then extends to an
action on $\op{m}$, $\clop$, and $\clopn$. We shall denote by
$\maK(\maH)$ the algebra of compact operators acting on a Hilbert
space $\maH$. Of course, we have $\clopn = \maK(L^2(M; E))$, since we
have considered only pseudodifferential operators with compactly
supported distribution kernels.

Let $S^*M$ denote the {\em unit cosphere bundle} of a smooth manifold
$M$, that is, the set of unit vectors in $T^*M$, as usual. We shall
denote, as usual, by $\maC_0(S^*M; \End(E))$ the set of continuous
sections of the {\em lift} of the vector bundle $\End(E) \to M$ to
$S^*M$.

\begin{corollary}\label{cor.full.symbol}
We have an exact sequence
\begin{equation*}
   0 \, \to \, \maK(L^2(M; E))^\Gamma \, \to \, \clop^\Gamma\,
   \stackrel{\sigma_{0}}{-\!\!\!\longrightarrow}\,
   \maC_0(S^*M; \End(E))^\Gamma \, \to \, 0 \,.
\end{equation*}
\end{corollary}

\begin{proof}
See, for instance, \cite[Corollary 2.7]{BCLN1}.
\end{proof}

\subsubsection{The structure of regularizing operators}
\label{sec.structure.of.regularizing.operators}
From now on, all our vector bundles will be $\Gamma$-equivariant
vector bundles. We want to understand the structure of the algebra
$\pi_\alpha(\clop^\Gamma)$, for any fixed $\alpha \in
\widehat{\Gamma}$ (see Equations \eqref{eq.restriction} and
\eqref{eq.restriction2} for the definition of the restriction morphism
$\pi_\alpha$ and of the projectors $p_\alpha \in C^*(\Gamma)$).

We shall need the following standard result about negative order
operators. Recall that, for $\alpha \in \widehat{\Gamma}$, we let
$\pi_\alpha$ be the representation of $\clop^\Gamma$ on $L^2(M;
E)_{\alpha}$ defined by restriction as before, Equations
\eqref{eq.restriction} and \eqref{eq.restriction2}.

\begin{proposition} \label{prop.image}
We have the identifications
\begin{align*}
  p_\alpha\psi^{-1}(M;E)^\Gamma & \simeq
  \pi_\alpha(\overline{\psi^{-1}}(M;E)^\Gamma) \\
  & = \pi_\alpha(\maK(L^2(M; E))^\Gamma) =
  \maK(L^2(M;E)_\alpha)^\Gamma\,,
\end{align*}
where the first isomorphism map is simply $\pi_\alpha$ and
\begin{equation*}
  \maK(L^2(M; E))^\Gamma = \clopn^\Gamma \simeq \oplus_{\alpha \in
    \widehat{\Gamma}}\maK(L^2(M; E)_\alpha)^\Gamma\,.
\end{equation*}
\end{proposition}

\begin{proof}
See, for example, \cite[Section 3]{BCLN1} for a proof.
\end{proof}

\section{The principal symbol}
\label{sec.principal.symbol}

From now on we assume that $M$ is \textbf{compact} and that $M/\Gamma$
is connected. Let us fix an irreducible representation $\alpha$ of
$\Gamma$ and consider the restriction morphism $\pi_\alpha$ to the
$\alpha$-isotypical component of $L^2(M;E)$. Recall that this morphism
was first introduced in Equation \eqref{eq.restriction} and discussed
in detail in Section \ref{ssec.group.r}. As in \cite{BCLN1}, we now
turn to the identification of the quotient
\begin{equation*}
  \pi_\alpha(\clop^\Gamma)/\pi_\alpha(\clopn^\Gamma).
\end{equation*}
The methods used in this paper diverge, however, drastically from
those of \cite{BCLN1}.

Since $\pi_\alpha(\clopn^\Gamma)$ was identified in the previous
section, the promised identification of the quotient
$\pi_\alpha(\clop^\Gamma)/\pi_\alpha(\clopn^\Gamma)$ will give further
insight into the structure of the algebra $\pi_\alpha(\clop^\Gamma)$
and will provide us, eventually, with Fredholm conditions. Recall
that, in this paper, we are assuming $\Gamma$ to be
finite. Nevertheless, a several intermediate results hold also in the
case $\Gamma$ compact.

\subsection{The primitive ideal spectrum of $A_M^\Gamma$} 
As before, $S^*M$ denotes the unit cosphere bundle of $M$. For the
simplicity of the notation, we shall write
\begin{equation*}
  A_M \ede \maC(S^*M;\End(E)),
\end{equation*}
as in the Introduction. Recall from Corollary \ref{cor.full.symbol}
that we have an algebra isomorphism
\begin{equation}\label{eq.psiso}
  \clop^\Gamma/ \clopn^\Gamma \simeq A_M^\Gamma.
\end{equation}
In our case, the inclusion $j :  \maC(S^*M/\Gamma) \subset
A_M^\Gamma$ as a central subalgebra induces, as in Equation
\eqref{eq.centr.char}, a central character map
\begin{equation*}
   j^* : \Prim (A_M^\Gamma) \to S^*M/\Gamma,
\end{equation*}
that underscores the local nature of the structure of the primitive
ideal spectrum of $A_M^\Gamma$.  We introduce the representation
$\pi_{\xi,\rho}$ defined for any $f \in A_M^\Gamma$ by
\begin{equation*}
  \pi_{\xi,\rho} (f) \seq \pi_\rho(f(\xi)),
\end{equation*}
that is $\pi_{\xi,\rho} (f)$ is the restriction of
$f(\xi)\in \End(E_x)$ to the $\rho$-isotypical component of $E_x$. The
central character map $j^*$ was used in \cite{BCLN1}, Corollary 4.2, to
obtain the following identification of $\Prim (A_M^\Gamma)$.

\begin{proposition}[\cite{BCLN1}]\label{prop.prim} 
Let $\Omega_M$ be the set of pairs $(\xi, \rho)$, where $\xi \in
S_x^*M$, $x \in M$, and $\rho \in \widehat \Gamma_\xi$ appears in
$E_x$ (i.e.\ $\Hom_{\Gamma_\xi}(\rho, E_x) \neq 0$).
  \begin{enumerate}
    \item The map $\Omega_M/\Gamma \ni \Gamma(\xi,\rho) \mapsto \ker
      (\pi_{\xi,\rho}) \in \Prim (A_M^\Gamma)$ is bijective.
    \item The central character map $\Omega_M/\Gamma \simeq \Prim
      (A_M^\Gamma) \to S^*M/\Gamma$ maps $\Gamma(\xi,\rho) \in
      \Omega_M/\Gamma$ to $\Gamma\xi$ and is continuous and
      finite-to-one.
  \end{enumerate}
\end{proposition}

The space $\Prim (A_M^\Gamma)$ is endowed with the Jacobson topology,
which was recalled in Subsection \ref{ssec.primitive.spec}; thus
Proposition \ref{prop.prim} allows us to use the central
character map $j^*$ to obtain a topology on $\Omega_M/\Gamma$ that
will play a crucial role in what follows. We thus now turn to the
study of this topology on $\Omega_M/\Gamma$. We begin with the
following standard lemma.

\begin{lemma} \label{lemma.topology}
Let $A$ be a $C^*$-algebra. The family $(V_a)_{a\in A}$ defined by
\begin{equation*}
    V_a = \{ J \in \Prim A \mid a \notin J \},
\end{equation*}
for any $a \in A$, is a basis of open sets for $\Prim (A)$.
\end{lemma}

\begin{proof}
Following \cite{Dixmier}, we know that the open, non-empty subsets of
$\Prim (A)$ are exactly the sets 
\begin{equation*}
  \{ J \in \Prim (A) \mid I \not\subset J \} \simeq \Prim(I)
\end{equation*}
where $I$ ranges through the closed, non-zero, two-sided ideals of
$A$. If $a \in A$, let us denote by $I_a := \overline{AaA}$ the
closed, two-sided ideal generated by $a$. Then $a \notin J
\Leftrightarrow I_a \not\subset J$, and hence $V_a = \Prim(I_a)$. This shows that $V_a$ is open.

Next, let $V \subset \Prim(A)$ be a non-empty open subset and
  $J_0 \in V$. We know then that there exists a closed, two-sided
  ideal $I$, $0 \neq I \subset A$, such that $V = \Prim(I)$. We
  have $I \not \subset J_0$, and hence we can choose $a \in I
  \smallsetminus J_0$. If $J \subset A$ is a primitive ideal such
that $a \notin J$, then {\em a fortiori} $I \not\subset J$. Therefore
$V_a \subset \Prim(I)$. This shows that $J_0 \in V_a \subset
  V$. Therefore the family $(V_a)_{a \in A}$ is a basis for the
topology on $\Prim (A)$.
\end{proof}

We shall use the bijection of Proposition \ref{prop.prim} to conclude
the following.

\begin{corollary} \label{cor.topology}
A basis for the induced topology on $\Omega_M/\Gamma \simeq \Prim
(A_M^\Gamma)$ is given by the sets
\begin{equation*}
    V_f \ede \{\, \Gamma(\xi,\rho) \in \Omega_M/\Gamma \, \mid\,
    \pi_{\xi,\rho}(f) \neq 0 \, \},
\end{equation*}
where $f$ ranges through the non-zero elements of $A_M^\Gamma$.
\end{corollary}

\subsection{The restriction morphisms}
Let $\maO \subset M$ be an open subset. Then $S^*\maO$ is the
restriction of $S^*M$ to $\maO$. We shall need the algebras
\begin{equation}\label{eq.def.ABZ}
  A_\maO \ede \maC_0(S^*\maO; \End(E))\ \ \mbox{ and } \ \ B_\maO \ede
  \overline{\psi^0}(\maO; E)\,.
\end{equation}

Assume that $\maO \subset M$ is $\Gamma$-invariant. The group $\Gamma$
does not act, in general, as multipliers on the $C^*$-algebra $B_\maO
\ede \overline{\psi^0}(\maO; E)$ (it does however act
by conjugation), so the method used in \cite{BCLN1} to compute
$\overline{\psi^{-1}}(\maO;E)^\Gamma \simeq \maK(L^2(\maO;E))^\Gamma$
does not extend to compute $B_\maO^\Gamma$. We shall thus consider the
natural, surjective map
\begin{multline}\label{eq.def.rho}
  \maR_\maO \, : \, A_\maO^\Gamma \ede \maC_0(S^*\maO; \End(E))^\Gamma
  \, \simeq \, B_\maO^\Gamma / \overline{\psi^{-1}}(\maO; E)^\Gamma \\
  \to
  \pi_\alpha(B_\maO^\Gamma)/\pi_\alpha(\overline{\psi^{-1}}(\maO; E)^\Gamma).
\end{multline}
Recall from Corollary \ref{prop.image} that
$\pi_\alpha(\overline{\psi^{-1}}(M;E)^\Gamma) =
\maK(L^2(M;E)_\alpha)^\Gamma$. Therefore, for a given $P \in \clop$,
we have that $\pi_\alpha(P)$ is Fredholm if, and only if, the
principal symbol of $P$ is invertible in $A_M^\Gamma/\ker
(\maR_M)$. This will be discussed in more detail in the next section.

We shall approach the computation of $\ker (\maR_M) \subset
A_M^\Gamma$ by determining the closed subset
\begin{equation}\label{eq.def.Xi}
  \Xi \ede \Prim(A_M^\Gamma/\ker(\maR_M)) \, \subset \,
  \Prim(A_M^\Gamma)
\end{equation}
of the primitive ideal spectrum of $A_M^\Gamma$ corresponding to
$\ker(\maR_M)$. Once we will have determined $\Xi$, we will also have
determined $\ker(\maR_M)$, in view of the definitions recalled in
Subsection \ref{ssec.primitive.spec} that put in bijection the closed,
two-sided ideals of a $C^*$-algebra with the closed subsets of its
primitive ideal spectrum.

Since $\maC(M/\Gamma) \subset B_M$, it follows from the
definition of $\maR_M$ that it is a $\maC(M/\Gamma)$--module morphism,
and hence that $\ker (\maR_M)$ is a $\maC(M/\Gamma)$--module. Let us
also recall that
\begin{equation*}
  \maC(M/\Gamma) \seq \maC(M)^\Gamma \, \subset\, Z_M \ede
  \maC(S^*M)^\Gamma\, \subset\, Z(A_M^\Gamma) \, \subset\, A_M^\Gamma
  \, \subset\, A_M\,.
\end{equation*}
The local nature of $\ker(\maR_M)$ and of the space $\Xi$ is explained in
the following remark.

\begin{remark}\label{rem.cover}
Let $M/\Gamma = \cup V_k$ be an open cover and 
\begin{equation*}
 \ker (\maR_M)_{V_k} :=
\maC_0(V_k) \ker (\maR_M) = \ker(\maR_{V_k}).
\end{equation*}
If we determine each
$\ker (\maR_M)_{V_k}$, then we determine $\ker (\maR_M)$ using a
partition of unity through:
\begin{equation}\label{eq.det.rho}
  \ker(\maR_M) \seq \ {\sum_{k}}' \phi_k \ker(\maR_{V_k})\,,
\end{equation} 
where $\sum'$ refers to sums with only finitely many non-zero terms
and $(\phi_k)$ is a partition of unity of $M/\Gamma$ with continuous
functions subordinated to the covering $(V_k)$ (thus, in particular,
$\supp(\phi_k) \subset V_k$). Since $M$ is compact, we can assume the
covering to be finite. (If $M$ was non-compact, then we would need to
take the closure of the right hand side in Equation
\eqref{eq.det.rho}.) To determine $\maR_M$, we can therefore replace
$M$ by any of the open sets $V_k$ in the covering and study
$\ker(\maR_{V_k})$. We shall do that for the covering of $M/\Gamma$
with the tubes $W_x \simeq \Gamma \times_{\Gamma_x} U_x$ considered in
\ref{ssec.slice}, see Equation \eqref{eq.def.tube}.
\end{remark}

\subsection{Local calculations}
In view of Remark \ref{rem.cover}, we shall concentrate now on the
local structure of $\ker(\maR_M)$, that is, on the structure of
$\ker(\maR_\maO)$ for suitable (``small'') open, $\Gamma$-invariant
subsets $\maO \subset M$. Let us fix then $x \in M$ and let $W_x
\simeq \Gamma \times_{\Gamma_x}U_x$ be the tube around $x$, Equation
\eqref{eq.def.tube}. For simplicity, we shall write
\begin{equation}\label{eq.def.Ax}
  A_x \ede A_{U_x} \ede \maC_0(S^*U_x; \End(E)) \ \mbox{ and } \ Z_x
  \ede Z(A_{x}^{\Gamma_x}) \,.
\end{equation}
For these algebras, the role of $\Gamma$ will be played by
$\Gamma_x$. For the statement of the following lemma, recall the
definitions in Subsection \eqref{ssec.group.actions}, especially
Equation \eqref{eq.def.tube}.

\begin{lemma}\label{lemma.local1}
Let $W_x \simeq \Gamma \times_{\Gamma_x} U_x$. Then $S^*W_x
\simeq\Gamma \times_{\Gamma_x} S^*U_x$ and we have
$\Gamma$-equivariant algebra isomorphisms
\begin{equation*}
  A_{W_x} \hspace*{-0.1cm}\ede \maC_0(S^*W_x; \End(E))\, \simeq \,
  \Ind_{\Gamma_x}^{\Gamma}\big (\maC_0(S^*U_x; \End(E)) \big )\,
  =: \hspace*{-0.1cm}\, \Ind_{\Gamma_x}^{\Gamma}(A_x)\,.
\end{equation*}
Consequently, the Frobenius isomorphism $\Phi$ of Equation
\eqref{eq.Frobenius2} induces an isomorphism
\begin{equation*}
  \Phi^{-1} \, : \, A_{W_x}^{\Gamma} \, \to \, A_x^{\Gamma_x} \,.
\end{equation*}
\end{lemma}

\begin{proof} We have that $E \vert_{W_x} \simeq \Gamma \times
_{\Gamma_x}(E\vert_{U_x})$, hence $\End(E) \vert_{W_x} \simeq
  \Gamma \times _{\Gamma_x}(\End(E)\vert_{U_x})$. Equation
  \eqref{eq.induced} then gives that $\maC_0(W_x, \End(E)) \simeq
  \Ind_{\Gamma_x}^\Gamma(\maC_0(U_x, \End(E)))$. The rest follows
  right away from the Frobenius reciprocity (more precisely, from
  Equation \eqref{eq.Frobenius2}) and from Equation
  \eqref{eq.induced}, with $\beta$ replaced with $\End(E)$.
\end{proof}

\begin{remark}
In view of Equation \eqref{eq.Frobenius2}, the isomorphism $\Phi$ of
Lemma \ref{lemma.local1} can be written explicitly as follows.  Let $f
\in A_x^{\Gamma_x}.$ Then, for any equivalence class $[\gamma, \xi]
:= \Gamma_x(\gamma, \xi) \in\Gamma\times_{\Gamma_x}
  S^*U_x  \simeq S^*W_x $ we have
\begin{equation*}
    \Phi(f)([\gamma, \xi]) \seq [\gamma, f(\xi)],
\end{equation*}
where $[\gamma , f(\xi)] \in \Gamma \times_{\Gamma_x} (U_x
  \times \End(E_x))^{\Gamma_x} \simeq \Gamma
  \times_{\Gamma_x} \End(E\vert_{U_x})^{\Gamma_x}
\simeq \End(E|_{W_x})^\Gamma$.
 This defines $\Phi(f) \in
\maC_0(S^*W_x; \End(E|_{W_x}))^\Gamma = A_{W_x}^\Gamma$.
\end{remark}

Lemma \ref{lemma.local1} together with the following remark will allow
us to reduce the study of the algebra $A_M^\Gamma$ to that of its
analogues defined for slices.

\begin{remark}\label{rem.local}
Let $U$ be an open set of some euclidean space and $W = U \times \{1,
2, \ldots, N\}$, where the space on the second factor is endowed with
the discrete topology. For simplicity, we identify $L^2(W)$ with
$L^2(U)^N$ using the map $f \mapsto (f(i))_{i=1\cdots N}$.
 Then
\begin{equation}
  \begin{gathered}
     \psi^{-1}(W) = M_N(\psi^{-1}(U))\simeq \psi^{-1}(U) \otimes
     M_N(\CC) \mbox{ and hence}\\
    \overline{\psi^{-1}}(W)=M_N(\overline{\psi^{-1}}(U)) \simeq
    \overline{\psi^{-1}}(U) \otimes M_N(\CC) \, .
  \end{gathered}
\end{equation}
On the other hand, if $A^N$ denotes the direct sum of $N$-copies of
the algebra $A$, then we have the following inclusions of algebras
\begin{equation}
  \begin{gathered}
     \psi^{0}(U)^N \subset \psi^{0}(W) \subset M_N(\psi^{0}(U)) \simeq
     \psi^{0}(U)\otimes M_N(\CC), \mbox{ and hence}\\
    \overline{\psi^{0}}(U)^N \subset \overline{\psi^{0}}(W) \subset
    M_N(\psi^{0}(U)) \simeq \overline{\psi^{0}}(U) \otimes M_N(\CC) \,
    .
  \end{gathered}
  \end{equation}
\end{remark}

The following lemma makes explicit the group actions in the
isomorphisms of the last remark. Thus, in analogy with the definitions
of the algebras $A_{W_x} = \maC_0(S^*W_x; \End(E))$ and $A_x =
\maC_0(S^*U_x; \End(E))$, we consider the algebras
\begin{equation}\label{eq.def.Bx}
  B_{W_x} \ede \overline{\psi^0}(W_x; E) \ \ \mbox{ and } \ \ B_x
  \ede \overline{\psi^0}(U_x; E) \,.
\end{equation}
We shall also use the standard notation $V^{(I)} := \{ f : I \to V \}$
for $I$ finite, as before.

\begin{lemma}\label{lemma.local2}
We keep the notation of Lemma \ref{lemma.local1} and of Equation
\eqref{eq.def.Bx} above. Then we have $\Gamma$-equivariant algebra
isomorphisms
\begin{equation*}
  B_{W_x} \, \simeq \, \Ind_{\Gamma_x}^{\Gamma}(B_x) +
  \overline{\psi^{-1}}(W_x; E) \,.
\end{equation*}
Consequently, $B_{W_x}^{\Gamma} \simeq \Phi(B_x^{\Gamma_x}) +
\overline{\psi^{-1}}(W_x; E)^{\Gamma}.$
\end{lemma}

\begin{proof}
Since $B_y = B_{U_y} \subset B_{W_x}$ for all $y \in \Gamma x$ and
since $U_x$ and $U_y$ are diffeomorphic through any $\gamma \in
\Gamma$ such that $\gamma x = y$ we obtain the inclusion
$B_x^{(\Gamma/\Gamma_x)} \subset B_{W_x}$, as in Remark
\ref{rem.local}. Similarly, since $B_x \to A_x$ is surjective, we
obtain the equality $B_{W_x} = B_x^{(\Gamma/\Gamma_x)} +
\overline{\psi^{-1}}(W_x; E)$ as in the same remark. From Equation
\eqref{eq.psiso} and Lemma \ref{lemma.local1} we know that
$B_{W_x}/\overline{\psi^{-1}}(W_x; E) \simeq A_{W_x} \simeq
A_{U_x}^{(\Gamma/\Gamma_x)} = \Ind_{\Gamma_x}^{\Gamma} (A_x)$, and
hence we obtain $B_{W_x} \simeq \Ind_{\Gamma_x}^{\Gamma}(B_x) +
\overline{\psi^{-1}}(W_x; E) \,.$ The last isomorphism follows from
the Frobenius reciprocity (more precisely, from Equation
\eqref{eq.induced}, with $\beta$ replaced with $B_x$) and from the
exactness of the functor $V \to V^{\Gamma}$.
\end{proof}

To be able to make further progress, it will be convenient to look
first at the case when $x \in M$ has minimal isotropy $\Gamma_x \sim
\Gamma_0$, that is, when $x$ belongs to the principal orbit bundle
$M_0 := M_{(\Gamma_0)}$. The notation $\Gamma_0$ will remain fixed
from now on.

\subsection{Calculations for the principal orbit bundle}
We assume as before that $M/\Gamma$ is connected.  Let $\Gamma_0$ be a
minimal isotropy group (which, we recall, is unique up to
conjugation). Let $x \in M$ be our fixed point and $\Gamma_x$ its
isotropy, as before. The case when $\Gamma_x$ is conjugated to
$\Gamma_0$ is simpler since, as noticed already, then $\Gamma_x$ acts
trivially on $U_x$.

Let us fix $x \in M$ with isotropy group $\Gamma_x = \Gamma_0$. As
before, we let
\begin{equation*}
   W_x \simeq \Gamma \times_{\Gamma_0} U_x \mbox{ and } E_{|W_x}\simeq 
   \Gamma\times_{\Gamma_0} (U_x \times \beta) \,,
\end{equation*}
where $\beta$ is some $\Gamma_0$-module, as in Equations
\eqref{eq.def.tube} and \eqref{eq.trivial.E}.  We decompose $\beta$
into a direct sum of representations of the form $\beta_j^{k_j}$ for
some non-isomorphic irreducible module (or representation) $\beta_j$
of $\Gamma_0$, again as before:
\begin{equation*}
    E_x \seq \beta \simeq \oplus \beta_j^{k_j}\,.
\end{equation*}

\begin{remark} \label{rem.principal.bundle.S^*M}
We have noticed earlier that $\Gamma_0$ acts trivially on $U_x$, hence
on $T^*_xM$. In particular $S^*M$ also has $\Gamma_0$ as minimal
isotropy subgroup, and $S^*M_{0}$ is a dense subset of the principal
bundle of $S^*M$.

\end{remark}

\begin{corollary}\label{cor.str2}
Let $x \in M$ be such that $\Gamma_x = \Gamma_0$ and $\beta =
\oplus_{j=1}^N \beta_j^{k_j}$, for some non-isomorphic, irreducible
$\Gamma_0$-modules $\beta_j$. Then
\begin{equation*}
   A_{W_x}^{\Gamma} \, \simeq \, A_{x}^{\Gamma_x}  \, \simeq \, \maC_0(S^*U_x)
   \otimes \End_{\Gamma_0} (\beta) \, \simeq \, \oplus_{j=1}^N M_{k_j}
\big (\maC_0(S^*U_x) \big ).
\end{equation*}
In particular, the canonical central character map
\begin{equation*}
 \Prim(A_x^{\Gamma_0}) \to S^*U_x \simeq
\Prim(\maC_0(S^*U_x)^{\Gamma_0})
\end{equation*}
of Proposition \ref{prop.prim}
corresponds to the trivial finite covering $S^*U_x \times
\Prim(\End_{\Gamma_0} (\beta)) \to S^*U_x$.
\end{corollary}

\begin{proof}
The first isomorphism is repeated from Lemma \ref{lemma.local1}.  The
second one is obtained from the following:
\begin{enumerate}[(i)]
\item from the definition of $A_x = A_{U_x}$,
\item from the assumption that $\Gamma_x = \Gamma_0$,
\item from the fact that $\Gamma_0$ acts trivially on $U_x$, and
\item from the identifications
\begin{equation*}
  A_x^{\Gamma_0} \ede \maC_0(S^*U_x; \End(E))^{\Gamma_0} \simeq
  \maC_0(S^*U_x) \otimes \End(\beta)^{\Gamma_0}.
\end{equation*}
\end{enumerate}

The last isomorphism follows from Example \ref{ex.group} and the
isomorphism $M_n(\CC) \otimes A \simeq M_n(A)$, valid for any algebra
$A$. The rest follows from Lemma \ref{lemma.product}.

Indeed, since both $\maC_0(S^*U_x)$ and $\End(\beta)^{\Gamma_0}$ have
only finite dimensional irreducible representations, we obtain
$\Prim(A_x^{\Gamma_0}) = S^*U_x \times \Prim(\End_{\Gamma_0}(\beta))
\simeq S^*U_x \times \{1, 2, \ldots, N\}$, where we use the
identification $\Prim(\maC_0(S^*U_x)) \simeq S^*U_x$ and where the set
$\{1, 2, \ldots, N\}$ is in natural bijection with the primitive ideal
spectrum of the algebra $\End_{\Gamma_0}(\beta) \simeq \oplus_{j=1}^N
M_{k_j}(\CC)$. The inclusion $\maC_0(S^*U_x) =
\maC_0(S^*U_x)^{\Gamma_0} \to A_x^{\Gamma_0}$ is given by the unital
inclusion $\CC \to \oplus_{j=1}^N M_{k_j}(\CC)$.  Hence the map
$\Prim(A_x^{\Gamma_0}) \to S^*U_x$ identifies with the first
projection in $S^*U_x \times \{1, 2, \ldots, N\} \to S^*U_x$.  That
is, it is a trivial covering, as claimed.
\end{proof}

The fibers of $\Prim(A_{M_0}^{\Gamma}) \to M_{0}/\Gamma$ are thus the
simple factors of $\End(E_x)^{\Gamma_0}$, whose structure was
determined in Example \ref{ex.group}.  We shall need the following
remark similar to Remark \ref{rem.local}, but simpler.

\begin{remark}\label{rem.local.bundle}
Let $U$ be an open subset of a euclidean space, let $V$ be a finite
dimensional vector space and let $V$ denote, by abuse of notation,
also the trivial, vector bundle with fiber $V$. Then we have {\em
  natural} isomorphisms
\begin{equation*}
  \begin{gathered}
    \psi^{-1}(U; V) \, \simeq\, \psi^{-1}(U) \otimes \End(V) \ \mbox{
      and}\\
    \psi^0(U; V) \, \simeq\, \psi^0(U) \otimes \End(V) \, .
  \end{gathered}
\end{equation*}
Consequently, we also have the analogous isomorphisms for the
completions
\begin{equation*}
  \begin{gathered}
    \overline{\psi^{-1}}(U; V) \, \simeq\, \overline{\psi^{-1}}(U)
    \otimes \End(V) \ \mbox{ and}\\
    \overline{\psi^0}(U; V) \, \simeq\, \overline{\psi^0}(U)
    \otimes \End(V) \, .
  \end{gathered}
\end{equation*}
\end{remark}

We are in position now to determine the kernel of $\maR_{W_x}$, when
$x$ is in the principal orbit bundle. We will use the notation of
Subsection \ref{ssec.group.actions} that was recalled at the beginning
of this subsection as well as the notation of Subsection
\ref{ssec.Frobenius}. In particular, recall that $\beta_j \in
\widehat{\Gamma}_0$ and $\alpha \in \widehat{\Gamma}$ are said to be
$\Gamma_0$-disjoint if $\beta_j$ is {\em not} contained in the
restriction of $\alpha$ to $\Gamma_0$. Also, $\Phi$ is the Frobenius
isomorphism, Equations \eqref{eq.Frobenius} and \eqref{eq.Frobenius2}
and Corollary \ref{cor.str2}.

\begin{proposition}\label{prop.str3}
Let $\Gamma_x = \Gamma_0$, let $E_x = \beta = \oplus_{j=1}^N
\beta_j^{k_j}$, and $\Phi : \maC_0(S^*U_x)
\otimes \End_{\Gamma_0}(\beta) \simeq A_x^{\Gamma_0} \to
A_{W_x}^{\Gamma}$ be the Frobenius isomorphism of Corollary
\ref{cor.str2}. Then
\begin{enumerate}
  \item $\maC_0(S^*U_x) \otimes \End_{\Gamma_0}(\beta_j^{k_j}) \subset
\Phi^{-1}(\ker(\maR_{W_x}))$ if $\beta_j$ and $\alpha$ are
$\Gamma_0$-disjoint, and
  \item $\maC_0(S^*U_x) \otimes \End_{\Gamma_0}(\beta_j^{k_j}) \cap
\Phi^{-1}(\ker(\maR_{W_x})) = 0$ if $\beta_j$ and $\alpha$ are
$\Gamma_0$-associated.
\end{enumerate}
In particular, Also, let $J \subset \{1, 2, \ldots, N\}$ be the set of
indices $j$ such that $\beta_j$ and $\alpha$ are $\Gamma_0$-disjoint,
then
\begin{equation*}
  \begin{gathered}
     \ker(\maR_{W_x}) \seq \Phi \big ( \oplus_{j \in J} \maC_0(S^*U_x)
     \otimes \End_{\Gamma_0}(\beta_j^{k_j}) \big) \ \mbox{ and}\\
  \pi_\alpha(B_M^\Gamma)/\pi_\alpha(\clopn^\Gamma) \, \simeq \, \Phi
  \big ( \oplus_{j \notin J} \maC_0(S^*U_x)
  \otimes \End_{\Gamma_0}(\beta_j^{k_j}) \big) \,.
  \end{gathered}
\end{equation*}
\end{proposition}

\begin{proof}
The proof is essentially a consequence of Proposition
\ref{prop.res.alpha} by including $U_x$ as a parameter, using also
Lemma \ref{lemma.local2}. To see how this is done, we will use the
notation of that lemma, in particular, $W_x \simeq
\Gamma\times_{\Gamma_0}U_x \simeq (\Gamma/\Gamma_0) \times U_x$ and $E
\simeq\Gamma\times_{\Gamma_0} (U_x \times \beta)$.  We identify $W_x$
with $\Gamma \times_{\Gamma_x}U_x $, i.e.\ we work with $W_x=\Gamma
\times_{\Gamma_x} U_x$.

Let $\pi_\alpha$ the fundamental morphism of restriction to the
$\alpha$-isotypical component, see Equations \eqref{eq.restriction}
and \eqref{eq.restriction2}. Recall that $B_x :=
\overline{\psi^0}(U_x; E)$.  Since $\Gamma_x$ acts trivially on $U_x$,
Remark \ref{rem.local.bundle} yields the $\Gamma$-equivariant
isomorphisms
\begin{equation}\label{eq.tensor}
  \Ind_{\Gamma_0}^\Gamma(B_x) \, \simeq \, \overline{\psi^0}(U_x)
  \otimes \Ind_{\Gamma_0}^\Gamma(\End(\beta)) \subset B_{W_x}\,,
\end{equation}
where the last inclusion is modulo the trivial identification given by
$P\otimes f(s)(\gamma ,x)=P(f(\gamma)s(\gamma))(x)$, $P\in
\overline{\psi^0}(U_x)$, $f\in \Ind_{\Gamma_0}^\Gamma(\End(\beta))$
and $s \in C_c(W_x, \End(E))$.  Combining further Remark
\ref{rem.local.bundle} with Remark \ref{rem.local}, we further obtain
the isomorphism
\begin{equation*}
  \overline{\psi^{-1}}(W_x;E) \, \simeq \, \overline{\psi^{-1}}(U_x)
  \otimes \End(\Ind_{\Gamma_0}^\Gamma(\beta)) \,.
\end{equation*}

Lemma \ref{lemma.local2} and the exactness of the functor $V \to
V^\Gamma$ give $\pi_{\alpha}(B_{W_x}^\Gamma) \seq \pi_{\alpha} \circ
\Phi (B_{x}^{\Gamma_x}) +
\pi_{\alpha}(\overline{\psi^{-1}}(W_x)^{\Gamma})$. Hence we obtain
\begin{equation*}
  \pi_{\alpha}(B_{W_x}^\Gamma)
  /\pi_{\alpha}(\overline{\psi^{-1}}(W_x)^{\Gamma}) = \pi_{\alpha} \circ
  \Phi (B_{x}^{\Gamma_x})/\pi_{\alpha} \circ \Phi (B_{x}^{\Gamma_x})
  \cap \pi_{\alpha}(\overline{\psi^{-1}}(W_x)^{\Gamma}).
\end{equation*}

Let $\mfkA$ and $\mfkJ$ be the image and, respectively, the kernel of
$\pi_\alpha: \Ind_{\Gamma_0}^\Gamma(\End(\beta))^\Gamma
\to \End(p_\alpha \Ind_{\Gamma_0}^\Gamma(\beta))$, which have been
identified in Proposition \ref{prop.res.alpha} in terms of the set
$J$. Recall next from Equation \eqref{eq.induced} that $L^2(W_x; E) =
L^2(U_x) \otimes \Ind_{\Gamma_0}^{\Gamma}(\beta)$, again
$\Gamma$-equivariantly. Each time, the action is on the second
component, since $\Gamma_0 = \Gamma_x$ acts trivially on
$\overline{\psi^0}(U_x)$. The action of $\Ind_{\Gamma_0}^\Gamma(B_x)
\subset B_{W_x}$ on $L^2(W_x; E) = L^2(U_x) \otimes
\Ind_{\Gamma_0}^{\Gamma}(\beta)$ is compatible with the tensor product
decomposition of Equation \eqref{eq.tensor}, in the sense that
$\overline{\psi^0}(U_x)$ acts on $L^2(U_x)$ and
$\Ind_{\Gamma_0}^\Gamma(\End(\beta))$ acts on
$\Ind_{\Gamma_0}^{\Gamma}(\beta)$. Also,
$\Ind_{\Gamma_0}^\Gamma(B_x)^\Gamma \simeq \overline{\psi^0}(U_x)
\otimes \Ind_{\Gamma_0}^\Gamma(\End(\beta))^\Gamma $, (we use this
isomorphism to identify them). We obtain that
\begin{equation}
 \pi_\alpha \circ \Phi (B_x^{\Gamma_x}) = \pi_\alpha(
     \Ind_{\Gamma_0}^\Gamma(B_x)^\Gamma)  \seq \overline{\psi^0}(U_x)
   \otimes \mfkA \,.
\end{equation}
On the other hand, Corollary \ref{prop.image} then gives that
$\pi_\alpha(\overline{\psi^{-1}}(W_x; \End(E))^\Gamma)$ is the algebra
of $\Gamma$-invariant compact operators acting on the space
$p_\alpha(L^2(W_x, \End(E)))$. Therefore, $\overline{\psi^{-1}}(U_x)
\otimes \mfkA \subset
\pi_\alpha(\overline{\psi^{-1}}(W_x; \End(E))^\Gamma)$, since
$\overline{\psi^{-1}}(U_x) \otimes \mfkA$ consists of compact,
$\Gamma$-invariant operators acting on
 $p_\alpha(L^2(W_x, E))$. Consequently,
\begin{multline}
   \overline{\psi^{-1}}(U_x) \otimes \mfkA \, \subset \,
   \pi_{\alpha}(\Ind_{\Gamma_0}^\Gamma(B_x)^\Gamma) \, \cap \,
   \pi_{\alpha}(\overline{\psi^{-1}}(W_x)^{\Gamma}) \\
   \subset \overline{\psi^0}(U_x) \otimes \mfkA \, \cap \,
   \maK(p_\alpha L^2(W_x; E))^\Gamma \, \subset \,
   \overline{\psi^{-1}}(U_x) \otimes \mfkA \,,
\end{multline}
and hence we have equalities everywhere.

Recall from Corollary \ref{cor.str2} that $A_{W_x}^\Gamma \simeq
A_x^{\Gamma_x}$. We obtain that the map
\begin{equation}
   \maR_{W_x} : A_{W_x}^\Gamma \, \simeq \,
   B_{W_x}^\Gamma/\overline{\psi^{-1}}(W_x; E)^\Gamma \to
   \pi_\alpha(B_{W_x}^\Gamma)/\pi_\alpha(\overline{\psi^{-1}}(W_x;
   E)^\Gamma)
\end{equation}
becomes, up to the canonical isomorphisms above, the map
\begin{equation}\begin{split}
  A_{x}^{\Gamma_x} \, \simeq \, \maC_0(S^*U_x)
  \otimes \End_{\Gamma_0}(\beta) \, \to \,
  \pi_{\alpha}(B_{W_x}^\Gamma)
  /\pi_{\alpha}(\overline{\psi^{-1}}(W_x)^{\Gamma}) \\
  = \pi_{\alpha} \circ \Phi (B_{x}^{\Gamma_x})/\pi_{\alpha} \circ \Phi
  (B_{x}^{\Gamma_x}) \cap
  \pi_{\alpha}(\overline{\psi^{-1}}(W_x)^{\Gamma})\\
  \simeq \overline{\psi^0}(U_x) \otimes
  \mfkA/\overline{\psi^{-1}}(U_x) \otimes \mfkA \, \simeq \,
  \maC_0(S^*U_x) \otimes \mfkA\,,
\end{split}
\end{equation}
with all maps being surjective and preserving the tensor product
decompositions. This identifies the kernel of $\maR_{W_x}$ with
$\maC_0(S^*U_x) \otimes \mfkJ$ and the image of $\maR_{W_x}$ with
$\maC_0(S^*U_x) \otimes \mfkA$. The rest of the statement follows from
the identification of $\mfkJ$ and $\mfkA$ in Proposition
\ref{prop.res.alpha}.
\end{proof}

Proposition \ref{prop.str3} above and its proof give the following
corollary.

\begin{corollary}\label{cor.subcover}
We use the notation of Proposition \ref{prop.str3} and we identify
the space $\Prim(\End(\beta))$ with $\{1, 2, \ldots, N\}$ as in Remark
\ref{ex.group}.  Then the homeomorphism $\Prim(A_{W_x}^{\Gamma})
\simeq S^*U_x \times \{1, 2, \ldots, N\}$ maps the set $\Xi \cap
\Prim(A_{W_x}^\Gamma)$ to $S^*U_x \times J$.  In particular, the
restriction $\Xi \cap \Prim(A_{W_x}^{\Gamma}) \to S^*U_x$ of the
central character is a covering as well.
\end{corollary}

\begin{proof}
Using the notations of the proof of Proposition \ref{prop.str3}, we
have that $\ker(\maR_{W_x})$ has primitive ideal spectrum $S^*U_x
\times \Prim(\mfkJ)$. We have $\Xi \cap \Prim(A_{W_x}^\Gamma) = S^*U_x
\times \Prim(\mfkA)$.
\end{proof}

The same methods yield the following result (recall that $M_0 =
M_{(\Gamma_0)}$ is the principal orbit bundle).

\begin{corollary}\label{cor.local2}
Let $M_0 := M_{(\Gamma_0)}$, the principal orbit bundle. The central
character map $\Prim(A_{M_0}^{\Gamma}) \to S^*M_{0}/\Gamma$ defined by
the inclusion $\maC_0(S^*M_0/\Gamma) \subset Z(A_{M_0}^{\Gamma})$ is a
covering with typical fiber $\Prim(\End(E_x)^{\Gamma_0})$ such that
$\Xi \cap \Prim(A_{M_0}^{\Gamma}) \to S^*M_{0}/\Gamma$ is a
subcovering, see \eqref{eq.def.Xi} for the definition of $\Xi$. In
particular, $\Xi \cap \Prim(A_{M_0}^{\Gamma})$ is open and closed in
$\Prim(A_{M_0}^{\Gamma})$.
\end{corollary}

\begin{proof}
The first statement is true locally, by Corollary \ref{cor.str2}, and
hence it is true globally. Indeed, let $x \in M_0$, let $ \xi \in
S_x^*M_0$, and let $\rho \in \widehat{\Gamma}_x$ that appears in $E_x$
(so $(\xi, \rho) \in \Omega_M$). We let $W_x \subset M_0 \subset M$ be
the typical tube with minimal isotropy $\Gamma_x = \Gamma_0$, as
before. Let $Z_{x} := \maC_0(S^*W_x)^\Gamma \subset Z_M =
\maC(S^*M)^\Gamma$. Then $\Prim(Z_x A_{M}^{\Gamma})$ is an open
neighborhood in $\Prim(A_{M_0}^{\Gamma})$ of the primitive ideal
$\ker(\pi_{\xi, \rho})$, see Proposition \ref{prop.prim} for notation
and details. We have that $Z_x A_{M}^{\Gamma} = A_{W_x}^{\Gamma}$ and
hence, on $\Prim(Z_x A_{M}^{\Gamma})$, the central character is a
covering, by Corollary \ref{cor.str2}. Similarly, its restriction to
$\Xi \cap \Prim(Z_x A_{M}^{\Gamma})$ is a covering by Corollary
\ref{cor.subcover}.
\end{proof}

Putting Corollary \ref{cor.local2} and Proposition \ref{prop.str3}
together we obtain the following results.

\begin{corollary}\label{cor.princ.open.set}
Let $M_0$ be the principal orbit type of $M$. The ideal
$\ker(\maR_{M_0}) = A_{M_0}^\Gamma \cap \ker(\maR_M)$ is defined by
the closed subset $\Xi_0 := \Xi \cap \Prim(A_{M_0}^{\Gamma})$ of
$\Prim(A_{M_0}^{\Gamma})$ consisting of the sheets of
$\Prim(A_{M_0}^{\Gamma}) \to S^*M_{0}/\Gamma$ that correspond to the
simple factors $\End(E_{x\rho})^{\Gamma_0}$ of
$\End(E_{x})^{\Gamma_0}$ with $\rho$ and $\alpha$
$\Gamma_0$-associated.
\end{corollary}

If $\Gamma$ is abelian, then $\rho$ and $\alpha$ are characters and
saying that they are $\Gamma_0$-associated means, simply, that their
restrictions to $\Gamma_0$ coincide: $\rho\vert_{\Gamma_0} =
\alpha_{\Gamma_0}$.  This is consistent with the definition given in
\cite{BCLN1}.

\subsection{The non-principal orbit case}
As in the rest of the paper, we assume $M/\Gamma$ to be connected. We
will show in Theorem \ref{theorem.non.principal} that $\Xi$ is the
closure of $\Xi_0$ in $\Prim (A_M^\Gamma)$. To that end, we first
construct a suitable basis of neighborhoods of $\Prim (A_M^\Gamma)$
using Lemma \ref{lemma.topology}.

\begin{remark} \label{rem.neighborhoods}
Let $\Gamma(\xi,\rho) \in \Prim (A_M^\Gamma)$, where we have used the
description of $\Prim(A_M^\Gamma)$ provided in Proposition
\ref{prop.prim} as orbits of pairs $\xi \in S^*M$ and suitable $\rho
\in \widehat{\Gamma}_\xi$. We construct a basis of neighborhoods
$(V_{\xi,\rho,n})_{n \in \NN}$ of $\Gamma(\xi,\rho)$ in $\Prim
(A_M^\Gamma)$ as follows. Let $\xi \in S_x^*M$ (that is, $\xi$ sits
above $x \in M$) and we use the notation $U_x$ and $W_x$ of Equation
\eqref{eq.def.tube}, as always.

First, by choosing a different point $\xi$ in its orbit, if necessary,
we may assume that $\Gamma_0 \subset \Gamma_\xi$. Now let $(\maO_n)_{n
  \in \NN}$ be a family of $\Gamma_\xi$-invariant neighborhoods of
$\xi$ in $S^*U_x$ such that:
\begin{itemize}
    \item for all $n$ and $\gamma \in \Gamma \smallsetminus\Gamma_\xi$, we
      have $\gamma \maO_n \cap \maO_n = \emptyset$,
    \item $\maO_{n+1} \subset \maO_n$ and $\bigcap_{n \in \NN} \maO_n
      = \{\xi\}$.
\end{itemize}
  
For any $n \in \NN$, we choose a function $\varphi_n \in
\maC_c(\maO_n)^{\Gamma_\xi}$ such that $\varphi_n \equiv 1$ on
$\maO_{n+1}$. Let $p_\rho \in \End(E_x)^{\Gamma_\xi}$ be the
projection onto $E_{x\rho}$. We can assume the bundle $E$ to be
trivial on $U_x$ and, using that, we first extend $p_\rho$ constantly
on $\maO_n$ and then as an element $q_n \in \maC_c(S^*U_x;\End
(E_x))^{\Gamma_x}$ defined as
\begin{equation*}
  q_n \ede \begin{cases}  \Phi_{\Gamma_\xi,\Gamma_x} (\varphi_n
    p_\rho) & \text{on $\Gamma_x \maO_n$} \\
  0 & \text{on $S^*U_x \smallsetminus \Gamma_x \maO_n$},
    \end{cases}
\end{equation*}
with $\Phi_{\Gamma_\xi,\Gamma_x}$ the Frobenius isomorphism of
Equation \eqref{eq.Frobenius2}. Let us set $\tilde{q}_n :=
\Phi_{\Gamma_x,\Gamma}(q_n) \in A_M^\Gamma$, where
$\Phi_{\Gamma_x,\Gamma}$ is the Frobenius isomorphism of Equation
\eqref{eq.Frobenius2}. Finally, we associate to $\tilde{q}_n$ the open
set
\begin{equation*}
    V_{\xi,\rho,n} \ede \{ J \in \Prim (A_M^\Gamma) \mid \tilde{q}_n
    \notin J \}.
\end{equation*}
Recall from \ref{lemma.topology} that $V_{\xi,\rho,n}$ is an open
subset of $\Prim (A_M^\Gamma)$. Moreover, it follows from our
definition that $V_{\xi,\rho,n+1} \subset V_{\xi,\rho,n}$ and that
$\bigcap_{n \in \NN} V_{\xi,\rho,n} = \{\Gamma(\xi,\rho)\}$.
\end{remark}

Recall that we are assuming that $M/\Gamma$ is connected.

\begin{theorem}\label{theorem.non.principal}
Let $\Xi := \Prim(A_M^\Gamma/\ker(\maR_M)) \subset
  \Prim(A_M^\Gamma)$ be the closed subset defined by the ideal
  $\ker(\maR_M)$. Then $\Xi$ is the closure in $\Prim(A_M^\Gamma)$ of
  the set $\Xi_0 := \Xi \cap \Prim(A_{M_0}^\Gamma)$, where $M_0$ is
  the principal orbit bundle of $M$.
\end{theorem}

\begin{proof}
We have that $\overline{\Xi}_0 \subset \Xi$ since $\Xi_0 \subset \Xi$
and the latter is a closed set. Conversely, let $\mathfrak{P} \in
\Prim (A_M^\Gamma) \smallsetminus \overline{\Xi}_0$. We will show that
$\mathfrak P \notin \Xi$. Let $\mathfrak P$ correspond to $(\xi ,
\rho) \in \Omega_M$, as in Proposition \ref{prop.prim}.  We may assume
that $\Gamma_0 \subset \Gamma_\xi$. Let $x$ be projection of $\xi$
onto $M$. Since the problem is local, we may also assume that $U_x
\subset T_x M$, that $M = W_x := \Gamma \times_{\Gamma_x} U_x$, and
that $E := \Gamma \times_{\Gamma_x} (U_x \times \beta)$ for some
$\Gamma_x$-module $\beta $.

Using the notations of Remark \ref{rem.neighborhoods}, there exists $n
> 0$ such that $V_{\xi, \rho, n} \cap \Xi_0 = \emptyset$. Let
$\tilde{q}_n = \Phi_{\Gamma_x,\Gamma}(q_n)$ be the symbol defined in
Remark \ref{rem.neighborhoods}. The description of $\Xi_0$ provided in
Corollary \ref{cor.princ.open.set}, the definition of $V_{\xi, \rho,
  n}$, and the definition of $\tilde q_n$ imply that $\pi_{ \zeta,
  \rho'}( \tilde q_n ) = 0$ for any $\zeta \in S^* M_{0}$ and $\rho'
\in \widehat{\Gamma}_0$ such that $\Gamma(\zeta, \rho') \in \Xi_0$,
that is, such that $\rho'$ and $\alpha$ are $\Gamma_0$-associated.

We next ``quantize $\tilde q_n$'' in an appropriate way, that is, we
construct an operator $\widetilde Q_n \in B_{W_x}^{\Gamma}$ with
symbol $\tilde q_n$ and with other convenient properties as follows.
First, let $\chi \in \maC_c^\infty(U_x)^{\Gamma_x}$ be such that $\chi
\varphi_n = \varphi_n$, which is possible since $\varphi_n$ has
compact support. Then let $\psi \in \maC^\infty(T^*_xM)^{\Gamma_x}$ be
such that $\psi(0) = 0$ if $|\eta| < 1/2$ and $\psi(\eta) = 1$
whenever $|\eta| \ge 1$. Recall that in this proof $U_x \subset T_xM$
is identified with its image in $M = \Gamma \times_{\Gamma_x} U_x$
through the exponential map. Let for any symbol $a$
\begin{equation*}
  Op(a) f(y) \ede \int_{T^*_xM}\int_{U_x} e^{i(y-z)\cdot \eta}\, a(y,
  z, \eta) f(z) dz d\eta .
\end{equation*}
We shall use this for $a_{n}(y,z, \eta) := \chi(y) \psi(\eta) \tilde
q_n\Big(\frac{\eta}{|\eta|}\Big)\chi(z)$, then set
\begin{equation*}
  \begin{gathered}
    Q_{n} \ede Op(a_{n}) \,, \quad \mbox{that is}\\
  Q_{n}f(y) \ede \int_{T^*_xM}\int_{U_x} e^{i(y-z)\cdot \eta}\, \chi(y)
  \psi(\eta) \tilde q_n \left(\frac{\eta}{|\eta|}\right)\chi(z) f(z)
  dz d\eta
  \end{gathered}
\end{equation*}
to be the standard pseudodifferential operator on $U_x$, associated to
the symbol $a_{n}(y, z, \eta) := \chi(y) \psi(\eta) \tilde
q_n\Big(\frac{\eta}{|\eta|}\Big)\chi(z)$. The operator $Q_{n}$ is
$\Gamma_x$-invariant by construction. Using the Frobenius isomorphism
of Equation \eqref{eq.Frobenius2}, we extend $Q_{n}$ to the operator
$\widetilde{Q}_n := \Phi (Q_{n})$, which acts on $M = W_x = \Gamma
\times_{\Gamma_x} U_x$ (see also Equation \eqref{eq.induced} with
regards to this isomorphism). Then $\widetilde{Q}_n \in
\Psi^0(M;E)^\Gamma$, that is, it is $\Gamma$-invariant, by
construction, and has principal symbol $\sigma_0(\widetilde{Q}_n) =
\tilde{q}_n$.

Now let $x_0 \in M_0 \cap U_x$, where, we recall, $M_0 :=
M_{(\Gamma_0)}$ denotes the principal orbit bundle. We have
\begin{equation*}
  L^2(W_{x_0};E) \seq \Ind_{\Gamma_0}^\Gamma \big ( L^2(U_{x_0};\beta)
  \big) \seq L^2\big (U_{x_0}; \Ind_{\Gamma_0}^\Gamma (\beta) \big),
\end{equation*}
where $\beta = E_{x_0} = E_x$ by the assumption that $E :=
  \Gamma \times_{\Gamma_x} (U_x \times \beta)$.

Let $\beta_j \in \widehat{\Gamma_0}$ be the isomorphism classes of the
$\Gamma_\xi$-submodules of $\beta$ and $k_j \ge 0$ is the dimension of
the corresponding $\beta_j$-isotypical component in $\beta$, so that
$\beta \simeq \oplus_{j=1}^N \beta_j^{k_j}$, as $\Gamma_0$-modules, as
before. Thus
\begin{equation*}
  L^2(W_{x_0};E) \, \simeq \, \bigoplus_{j = 1}^N L^2(U_{x_0};
  \Ind_{\Gamma_0}^\Gamma (\beta_j^{k_j})) \,.
\end{equation*}

Recall that the $\alpha$-isotypical component of
$\Ind_{\Gamma_0}^\Gamma (\beta_j^{k_j})$ is $\alpha \otimes
\Hom_{\Gamma}(\alpha, \Ind_{\Gamma_0}^\Gamma (\beta_j^{k_j}))$, which
is non-zero if, and only if, $\alpha$ and $\beta_j$ are
$\Gamma_0$-associated, by the Frobenius isomorphism. Hence, passing to
the $\alpha$-isotypical components, we have
\begin{equation} \label{eq.local_isotypical}
  L^2(W_{x_0};E)_\alpha \ede p_{\alpha} L^2(W_{x_0};E) =
  \bigoplus_{j \in J^c} L^2(U_{x_0}; \Ind_{\Gamma_0}^\Gamma
  (\beta_j^{k_j}))_\alpha \, ,
\end{equation}
where $J \subset \{1,\ldots,N\}$ is the set of indices such that
$\beta_j \in \widehat{\Gamma}_0$ and $\alpha$ are $\Gamma_0$-disjoint;
$J^c$ is its complement (i.e.\ $\beta_j \in \widehat{\Gamma}_0$ and
$\alpha$ are $\Gamma_0$-associated).

Let $p_J \in \End(\beta)^{\Gamma_0}$ be the projector onto
$\bigoplus_{j \in J^c} \beta_j^{k_j}$.  Recall that
$\pi_{\zeta,\beta_j}(\tilde{q}_n) = 0$ for any $(\zeta,\beta_j) \in
S^*M_0\times \widehat{\Gamma}_0$ with $j \notin J$. Therefore
$\tilde{q}_n(\zeta)p_J = 0$, for all $\zeta \in S^*M_0$. Since
$S^*M_0$ is dense in $S^*M$, this implies that $\tilde{q}_np_J =
0$. Thus
\begin{equation*}
  \widetilde{Q}_n p_J \seq Op(\chi \psi \tilde q_n \chi ) p_J \seq
  Op(\chi \psi \tilde q_n \chi p_J ) \seq 0 \,.
\end{equation*}
Hence for any $f \in L^2(W_{x_0};E)_{\alpha}$, we have that
$\widetilde{Q}_n f = 0$. This is true for any $x_0 \in M_0$, so we
conclude that $\widetilde{Q}_n$ is zero on $L^2(M_0;E)_\alpha$. Since
$M_0$ has measure zero complement in $M$, we have $L^2(M_0;E)_\alpha =
L^2(M;E)_\alpha$; therefore $\pi_\alpha(\widetilde{Q}_{n}) = 0$. This
implies that $\maR_M(\widetilde{q}_{n}) = 0$, while
$\pi_{\xi,\rho}(\widetilde{q}_{n}) = 1$. Thus $\Gamma(\xi,\rho) \notin
\Xi$, which concludes the proof.
\end{proof}

Our question now is to decide whether some given $\Gamma(\xi, \rho)$
is in $\Xi$ or not. Recall that $\rho$ and $\alpha$ are said to be
$\Gamma_0$-associated if $\Hom_{\Gamma_0}(\rho,\alpha) \neq 0$.  The
set $X_{M,\Gamma}^\alpha$ was defined in the introduction as the set
of pairs $(\xi,\rho)\in T^*M\smallsetminus\{0\}\times\widehat{\Gamma}_\xi$
for which there is an element $g \in \Gamma$ such that $g\cdot \rho$
and $\alpha$ are $\Gamma_0$-associated.

\begin{remark} \label{rem.null_isotypical_component}
Let us highlight the following interesting fact, implied by the proof
of Theorem \ref{theorem.non.principal}. We have that $E_{x\rho} = 0$
for any $(\xi,\rho) \in X^\alpha_{M_0,\Gamma}$ (with $x$ the
projection of $\xi$ on $M_0$) if, and only if, $L^2(M;E)_\alpha = 0$.

Indeed, for any $x\in M_0$ with $\Gamma_x = \Gamma_0$, we have noted
in Equation \eqref{eq.local_isotypical} that
\begin{equation*}
  L^2(W_x;E)_\alpha = \bigoplus_{\rho}
  L^2(U_x;\Ind_{\Gamma_0}^\Gamma(E_{x\rho}))_\alpha,
\end{equation*}
where the direct sum is indexed by the representations $\rho \in
\widehat{\Gamma}_0$ that are $\Gamma_0$-associated to $\alpha$. If $E_{x\rho} =
0$ for any such representation, then $L^2(W_x;E)_\alpha = 0$. Such open sets
$W_x$ cover $M_0$, so $L^2(M_0;E)_\alpha = 0$. Since $M_0$ has measure zero
complement, we conclude that $L^2(M;E)_\alpha= 0$. 
\end{remark}

\begin{proposition} \label{prop.computation.Xi}
We use the notation in the last two paragraphs. We have $\Gamma(\xi,
\rho) \in \Xi$ if, and only if, there is a $g \in \Gamma$ such that
$g\cdot\rho$ and $\alpha$ are $\Gamma_0$-associated.  
\end{proposition}

\begin{proof}
Let $\Gamma(\xi,\rho) \in \Prim (A_M^\Gamma)$, with $x\in M$ the base
point of $\xi$. We can assume (by choosing a different element in the
orbit if needed) that $\Gamma_0 \subset \Gamma_\xi$. Let $\tilde{q}_n
\in A_M^\Gamma$ be the element defined in Remark
\ref{rem.neighborhoods} and $V_{\xi,\rho,n}$ the corresponding
neighbourhood of $\Gamma(\xi,\rho)$ in $\Prim (A_M^\Gamma)$.

There is a $\Gamma_x$-equivariant isomorphism $E\vert_{U_x} \simeq U_x
\times \beta$, where $\beta = E_x$ is a $\Gamma_x$-module. Since
$\Gamma_0 \subset \Gamma_x$, we may decompose $\beta$ into
$\Gamma_0$-isotypical components, i.e.\ $\beta = \bigoplus_{j=1}^N
\beta_j^{k_j}$, with the usual notation. If $\eta \in \maO_n$, then
$\pi_{\eta,\beta_j}(\tilde{q}_n) =
\varphi_n(\eta)\pi_{\beta_j}(p_\rho)$.  Therefore, for any $\eta \in
S^*M$, we have
\begin{equation*}
    \pi_{\eta,\beta_j}(\tilde{q}_n) = 0 \Leftrightarrow
    \text{$\Hom_{\Gamma_0}(\beta_j,\rho) = 0$ or $\tilde{q}_n(\eta) =
      0$}.
\end{equation*}

This implies that
\begin{equation*}
    V_{\xi,\rho,n} \cap \Xi_0 = \{\Gamma(\eta,\beta) \in \Xi_0 \mid
    \text{$\tilde{q}_n(\eta) \neq 0$ and $\Hom_{\Gamma_0}(\beta,\rho)
      \neq 0$}\}
\end{equation*}
It follows from the determination of $\Xi_0$ in Corollary
\ref{cor.princ.open.set} that $V_{\xi,\rho,n} \cap \Xi_0 \neq
\emptyset$ if, and only if, we have $\Hom_{\Gamma_0}(\rho,\alpha) \neq
0$. Now $\Xi = \overline{\Xi}_0$ by Theorem
\ref{theorem.non.principal}. Since the open sets $(V_{\xi,\rho,n})_{n
  \in \NN}$ form a basis of neighborhoods of $\Gamma(\xi,\rho)$, we
conclude that $\Gamma(\xi,\rho) \in \Xi$ if, and only if, we have
$\Hom_{\Gamma_0}(\rho,\alpha) \neq 0$.
\end{proof}

\begin{remark} \label{rem.Xi=X_M_Gamma}
Our definition of $\alpha$-ellipticity for an operator $P \in
\clop^\Gamma$ was stated in terms of the set $X_{M,\Gamma}^\alpha$,
defined in Equation \eqref{eq.def.Xalpha}.  Proposition
\ref{prop.computation.Xi} establishes that $\Xi \simeq
\tilde{X}^\alpha_{M,\Gamma}/\Gamma$, where
$\tilde{X}^\alpha_{M,\Gamma}$ is the (possibly smaller) subset
of pairs $(\xi,\rho) \in X^\alpha_{M,\Gamma}$ such that $E_{x\rho}
\neq 0$ (with $x$ the projection of $\xi$ on $M$).  Keeping in mind
the fact that the null operator on a trivial vector space is
invertible, we have that $\sigma_0^\Gamma(P)(\xi,\rho)$ is invertible
for any $(\xi,\rho) \in X_{M,\Gamma}^\alpha$ if, and only if, it is
invertible for any $(\xi,\rho) \in \tilde{X}_{M,\Gamma}^\alpha$.  The
pathological case $\Xi = \emptyset$, for which $E_{x\rho} = 0$ for any
$(\xi,\rho) \in X_{M,\Gamma}^\alpha$, causes no problem: indeed, as
noticed in Remark \ref{rem.null_isotypical_component}, we then have
$L^2(M;E)_\alpha = 0$. In that case $\pi_\alpha(P)$ is Fredholm for
any $P \in \clop^\Gamma$, which is consistent with the invertibility
of $\sigma_0^\alpha(P)(\xi,\rho) : 0 \to 0$ for any $(\xi,\rho) \in
X_{M,\Gamma}^\alpha$.
\end{remark}

 We summarize part of the above discussions in the following
  proposition.
\begin{proposition}\label{prop.Xi=X_M_Gamma}
Let $\tilde X^\alpha_{M, \Gamma}$ be as in Remark
\ref{rem.Xi=X_M_Gamma}.  The primitive ideal spectrum $\Xi =
\Prim(A_M^\Gamma/\ker(\maR_M))$ is canonically homeomorphic to $\tilde
X^\alpha_{M, \Gamma}/\Gamma$ via the restriction map from $A_M^\Gamma
:= \maC(S^*M; \End(E))^\Gamma$ to sections over $X^\alpha_{M,
  \Gamma}$.
\end{proposition}

\section{Applications and extensions}
\label{sec.applications}

We now prove the main result of the paper, Theorem \ref{thm.main1}, on
the characterization of Fredholm operators, and discuss some
extensions of our results. We first explain how to reduce the proof to
the case $M/\Gamma$ connected and we discuss in more detail the
$\Gamma$-principal symbol and $\alpha$-ellipticity (this discussion
can be skipped at a first lecture).

\subsection{Reduction to the connected case and $\alpha$-ellipticity}
\label{ssec.disconnected}
In this subsection, unlike most of the rest of the paper, we {\em do
  \underline{not} assume that $M/\Gamma$ is connected} in order to
explain how to reduce the general case to the connected one. We do
assume however, as always, that $M$ is compact. We also provide some
other reductions of our proof.

Let $\pi_{M, \Gamma} : M \to M/\Gamma$ be the quotient map and let us
write then $M/\Gamma = \cup_{i=1}^N C_i$ as the {\em disjoint} union
of its connected components. We let $M_i := \pi_{M, \Gamma}^{-1}(C_i)$
be the preimages of these connected components. Note that, in
general, the submanifolds $M_i$ are not connected, but, for each $i$,
$M_i/\Gamma = C_i$ {\em is connected}. In particular, a similar
discussion applies to yield the definition of the space
\begin{equation}\label{eq.def.Xa.gen}
  X^\alpha_{M, \Gamma} \ede \sqcup_{i=1}^N X^\alpha_{M_i, \Gamma}\,
\end{equation}
as a disjoint union of the spaces $X^\alpha_{M_i, \Gamma}$, which
makes sense since each of the spaces $M_i$ is invariant for $\Gamma$
and $M_i/\Gamma$ is connected. (See Equation \eqref{eq.def.Xalpha} of the
Introduction for the definition of the spaces $X^\alpha_{M_i,
  \Gamma}$.)

We shall decorate with the index $i$ the restrictions of objects on
$M$ to $M_i$. Thus, $E_i := E\vert_{M_i}$, and so on and so
forth. This almost works for an operator $P \in \clop^\Gamma$. Indeed,
we first notice that
\begin{equation}\label{eq.dir.sum}
  L^2(M; E) \, \simeq \, \oplus_{i=1}^N L^2(M_i; E_i) \quad \mbox{ and
  } \quad
  \oplus_{i=1}^N \overline{\psi^{0}}(M_i; E_i) \, \subset \, \clop\,.
\end{equation}

Recall that $\maK(V)$ denotes the algebra of compact operators on a
Hilbert space $V$. The following proposition provides the desired
reduction to the connected case.

\begin{proposition}\label{prop.reduction}
Let $p_i : L^2(M; E) \to L^2(M_i; E_i)$ be the canonical orthoghonal
projection. For $P \in \clop$, we let $P_i := p_i P p_i \in
\overline{\psi^{0}}(M_i; E_i)$. Then $P - \sum_{i=1}^N P_i \in
\maK(L^2(M; E))$. If we regard $\sum_{i=1}^N P_i = \oplus_{i=1}^N P_i
$ as an element of $\oplus_{i=1}^N \overline{\psi^{0}}(M_i; E_i)$,
then we see that
\begin{equation*}
   \clop \seq \oplus_{i=1}^N \overline{\psi^{0}}(M_i; E_i) +
   \maK(L^2(M; E)) \,.
\end{equation*}
Moreover, $\pi_\alpha (P) - \oplus_{i=1}^N \pi_\alpha(P_i)$ is compact
and hence $\pi_\alpha (P)$ is Fredholm if, and only if, each
$\pi_\alpha(P_i)$ is Fredholm, for $i = 1,\ldots,N$.
\end{proposition}

\begin{proof}
If $i \neq j$, $p_iPp_j$ has zero principal symbol, and hence it is
compact. Therefore $P - \sum_{i=1}^N P_i = \sum_{i \neq j} p_iPp_j$ is
compact. The rest follows from Equation \eqref{eq.dir.sum}, its
corollary $L^2(M; E)_\alpha \, \simeq \, \oplus_{i=1}^N L^2(M_i;
E_i)_\alpha$, and the fact that $\pi_\alpha$ respects these direct sum
decompositions.
\end{proof}

Recall the algebras $A_{\maO}$ of symbols of the previous section,
see Equation \eqref{eq.def.ABZ}, where $\maO$ is an open subset of
$M$.

\begin{remark}\label{rem.symbol}
The $\Gamma$-principal symbol $\sigma^\Gamma_m(P)$ was defined in
\eqref{eq.def.Gamma.symb}, and we stress that the definition of the
space $X_{M,\Gamma}$ did not require that $M/\Gamma$ be connected. The
disjoint union definition of the space $X_{M,\Gamma} = \sqcup_{i=1}^N
X_{M_i, \Gamma}$ means that
\begin{equation*}
  \sigma_m^\Gamma(P)|_{X_{M_i,\Gamma}} = \sigma_m^\Gamma(P_i)
\end{equation*}
for each $i = 1,\ldots,N$. The analogous disjoint union decomposition
of $X^\alpha_{M, \Gamma} \ede \sqcup_{i=1}^N X^\alpha_{M_i, \Gamma}$
gives that $P$ is $\alpha$-elliptic if, and only if, for each $i$,
$P_i$ is $\alpha$-elliptic.

This allows us to reduce the proof of our main theorem, Theorem
\ref{thm.main1} to the connected case since, assuming that the connected case
has been proved, we have
\begin{align*}
  \pi_\alpha(P) \text{ is Fredholm} & \Leftrightarrow \forall i, \pi_\alpha(P_i)
  \text{ is Fredholm} \\
                                    & \Leftrightarrow \forall i, P_i \text{ is
                                    $\alpha$-elliptic} \\
                                    & \Leftrightarrow P \text{ is
                                    $\alpha$-elliptic},
\end{align*}
where the first equivalence is by Proposition \ref{prop.reduction},
the second equivalence is by the assumption that our main theorem has
been proved in the connected case, and the last equivalence is by the
first part of this remark.
\end{remark}
\color{black}

We now {\em resume our assumption that $M/\Gamma$ is connected,} for
convenience.  In particular, $\Gamma_0$ will be a minimal isotropy
group, which is unique up to conjugacy (since we are again assuming
that $M/\Gamma$ is connected). We shall take a closer look next at the
$\Gamma$- and $\alpha$-principal symbols, so the following simple
discussion will be useful. Recall that if $K \subset \Gamma$, $\rho
\in \widehat{\Gamma}$, and $g \in \Gamma$, then $g\cdot K := g K
g^{-1}$ and $(g \cdot \rho)(\gamma) := \rho( g^{-1} \gamma g)$, so
that $g \cdot \rho$ is an irreducible representation of $g \cdot K$
(i.e. $g \cdot \rho \in \widehat{g \cdot K}$).

\begin{remark}\label{rem.alpha.elliptic}
  Let $\xi \in T^*M \smallsetminus \{0\}$ and $\rho \in \widehat{\Gamma}_\xi$
(that is, $(\xi, \rho) \in X_{M, \Gamma}$). Then the following three
statements are equivalent:
\begin{enumerate}[(i)]
  \item the pair $(\xi, \rho) \in X^\alpha_{M, \Gamma}$;
\item there is $g \in \Gamma$ such that $\Gamma_0 \subset g \cdot
  \Gamma_\xi = \Gamma_{g\xi}$ and such that $g\cdot \rho$ and
  $\alpha$ are $\Gamma_0$-associated;

\item There is $\gamma \in \Gamma$ such that $\gamma \cdot \Gamma_0
  \subset \Gamma_\xi$ and $\Hom_{\gamma \cdot \Gamma_0}(\rho, \alpha)
  \neq 0$.
\end{enumerate}

Indeed, if (i) is satisfied, then the definition of $X^\alpha_{M,
  \Gamma}$ (see Equation \eqref{eq.def.Xalpha} and Definition
\ref{def.associated}) is equivalent to the existence of $g$, i.e.\ (i)
$\Leftrightarrow$ (ii). Recalling that $g\cdot \rho \in
\widehat{\Gamma_{g \xi}}$, we stress then that we need $\Gamma_0
\subset g \cdot \Gamma_\xi = \Gamma_{g\xi}$ for $\alpha$ and $g \cdot
\rho$ to be associated.

To prove (ii) $\Leftrightarrow$ (iii), let $\gamma = g^{-1}$.  We have
$\Gamma_0 \subset g \cdot \Gamma_\xi$ and $\Hom_{\Gamma_0}(g \cdot
\rho, \alpha) \neq 0$ if, and only if, $\gamma \cdot \Gamma_0 \subset
\Gamma_\xi$ and $\Hom_{\gamma \cdot \Gamma_0}(\rho, \gamma \cdot
\alpha) \neq 0$. The result follows from the fact that $\alpha$ and
$\gamma \cdot \alpha$ are equivalent (since $\gamma \in \Gamma$ and
$\alpha$ is a representation of $\Gamma$).
\end{remark}

We include next below, in Proposition
\ref{prop:fixed_points_reformulation}, a reformulation of our
$\alpha$-ellipticity condition in terms of the fixed point manifold
$S^*M^{\Gamma_0}$, with $\Gamma_0$ a minimal isotropy subgroup as
before. This result was suggested by some discussions with
P.-E.\ Paradan, whom we thank for his useful input.

In the following, $\operatorname{Stab}_\Gamma(M)$ will denote the set
of stabilizer subgroups $K$ of $\Gamma$, that is, the set of subgroups
$K \subset \Gamma$ such that there is $m\in M$ with $K=\Gamma_m$. It
is a finite set, since $\Gamma$ is finite. Similarly, we let
\begin{equation*}
 \operatorname{Stab}_\Gamma^{\Gamma_0}(M):=\{K\in
\operatorname{Stab}_\Gamma(M) \mid \Gamma_0 \subset K\}.
\end{equation*}
Note that $\operatorname{Stab}_\Gamma(T^*M) =
\operatorname{Stab}_\Gamma(M)$. Recall also that $(T^*M)^K =
T^*(M^K)$, where $M^K$ is the submanifold of fixed points of $M$ by
$K$, as usual.  For a $\Gamma$--space $X$ and $K \subset \Gamma$ a
subgroup, we shall let $X_{K} := \{x \in X \, \vert \ \Gamma_x = K\}
\subset X^K$ denote the set of points of $X$ with isotropy $K$. Note
that, in general, $T^*(M_{K}) \neq (T^*M)_{(K)}$.

\begin{lemma} \label{lm:stabilizer_manifold}
The set $M_K := \{m \in M \, \vert \ \Gamma_m=K\}$ is a submanifold.
\end{lemma}

\begin{proof}
Let $x \in M_K$, that is, $\Gamma_x = K$. The problem is local, so,
using, \cite[Proposition 5.13]{tomDieckTransBook}, we see that it
suffices to consider the case $M=\Gamma\times_K V$, where $V$ is a
$K$-representation. Then, if $z = (\gamma, y) \in \Gamma \times_K V$,
we have $\Gamma_z = \gamma K_y \gamma^{-1}$ and hence, if $\Gamma_z =
K$, we obtain $K = \gamma K_y \gamma^{-1}$, which, in turn, gives $K_y
= K$ and $\gamma \in N(K) := \{g \in \Gamma\,\vert \ g K g^{-1} = K \,
\}$. We thus obtain that
\begin{equation*}
  M_K \seq \{ (\gamma, y) \in \Gamma \times_K M \, \vert \ K_y =
  \gamma^{-1} K \gamma \, \} \seq N(K) \times_K V^K\,,
\end{equation*}
which is a submanifold of $M$.
\end{proof}

Let $K \subset \Gamma$ be a subgroup and $\rho \in \widehat{K}$. Then
$E_\rho \ede \bigsqcup_{x \in M^K} E_{x\rho}$ is a smooth vector
bundle over $M^K$, the set of fixed points of $M$ with respect to
$K$. Similarly, $(E \otimes \rho)^K \to M^K$ is a smooth vector bundle
(over $M^K$). Moreover, we have an isomorphism
\begin{equation}\label{eq.iso.K}
  \End(E_\rho)^K \simeq \End( (E\otimes \rho)^K \otimes \rho)^K \simeq
  \End((E\otimes \rho)^K),
\end{equation}
of vector bundles over $M^K$, where the last isomorphism comes from
the fact that $\End(\rho)^K = \CC$.  In view of this discussion, we
choose to state the following result in terms of the vector bundle
$(E\otimes \rho)^K$ over $M^K$ rather than in terms of $E_\rho$. This
discussion shows also that it is enough in our proofs to assume that
$\alpha$ is the trivial (one-dimensional) representation.

\begin{proposition} \label{prop:fixed_points_reformulation}
Let $\alpha \in \widehat{\Gamma}$ and $P \in \psi^m(M;E)$, for some $m
\in \RR$. Recall the vector bundle $(M \otimes \rho)^K \to M^K \supset
M_K$. The following are equivalent:
\begin{enumerate}
\item $P$ is $\alpha$-elliptic (Definition \ref{def.chi.ps}).
\item For all $K \in \operatorname{Stab}_\Gamma^{\Gamma_0}(M)$ and all
  $\rho \in \widehat{K}$ that are $\Gamma_0$-associated with $\alpha$,
  we have that $(\sigma_m(P) \otimes id_\rho) \vert_{(E\otimes
    \rho)^{K}}$ defines by restriction an invertible element of
\begin{equation*}
     \maC^\infty \big((T^*M \smallsetminus \{ 0\} )_K, \End ((E\otimes
     \rho)^K)\big) \,.
\end{equation*}

\item The principal symbol $(\sigma_m(P) \otimes
  id_\alpha)|_{(E\otimes\alpha)^{\Gamma_0}}$ defines by restriction an
  invertible element in
      \begin{equation*}
        \maC^\infty (T^*M^{\Gamma_0} \smallsetminus \{ 0\}
        ; \End((E\otimes \alpha)^{\Gamma_0})\,.
      \end{equation*}
\end{enumerate}  
\end{proposition}

Recall that for representations $\alpha$ and $\beta$ to be
$H$-associated, they have to be defined, after restriction, on
$H$. See Definition \ref{def.associated}.

\begin{proof}
Recall that $P$ is $\alpha$-elliptic if the restriction of
$\sigma_m^\Gamma(P)$ to $X^\alpha_{M, \Gamma}$ is invertible (see
Remark \ref{rem.alpha.elliptic} for a detailed definition and
discussion of the space $X^\alpha_{M, \Gamma}$ appearing in the
definition of $\alpha$-ellipticity).

Let $K \in \operatorname{Stab}_{\Gamma}^{\Gamma_0}(M)$ (so $\Gamma_0
\subset K$), $\rho \in \widehat{K}$, and $\xi \in T_x^*M
\smallsetminus\{0\}$ with $\Gamma_\xi = K$. We have that $(\sigma_m(P)
\otimes id_\rho)\vert_{(E \otimes \rho)^K}$ is invertible at $\xi \in
(T^*M)_K$ if, and only if, the restriction of $\sigma_m(P)(\xi)$ to
$E_{x \rho}$ is invertible, since they correspond to each other under
the isomorphism of Equation \eqref{eq.iso.K}. The relation (2) thus
means that the restriction of the principal symbol $\sigma_m(P)$ is
invertible on a {\em subset} of $X^{\alpha}_{M, \Gamma}$, so (1)
implies (2) right away.

Let us show next that (2) implies (1), let $\xi \in T^*M\smallsetminus
\{0\} $ and let $K':=\Gamma_\xi$. By definition, $\xi$ belongs to
$(T^*M)_{K'}$. Assume that $(\xi, \rho) \in X^{\alpha}_{M,
  \Gamma}$. This means that there exists $g \in \Gamma $ such that
$\rho' := g \cdot \rho$ and $\alpha$ are $\Gamma_0$ associated (see
Equation \eqref{eq.def.Xalpha} and Definition \ref{def.associated};
alternatively, this is also recalled in Remark
\ref{rem.alpha.elliptic}). For this to make sense, it is implicit that
\begin{equation*}
   \Gamma_0 \subset \Gamma_{g \xi} \seq g\cdot \Gamma_{\xi} \seq g
   \cdot K' \, =:\, K
\end{equation*}
(again, see Remark \ref{rem.alpha.elliptic}). Then $g : (T^*M)_{K'}
\rightarrow (T^*M)_K$ is a diffeomorphism. Condition (2) for the group
$K$ gives that $\pi_{g \xi, \rho'}(\sigma_m(P))$ is invertible, since
the irreducible representation $\rho'$ of $\Gamma_{g\xi}$ is
$\Gamma_0$-associated to $\alpha$ (we have used here again the
isomorphism \eqref{eq.iso.K}).  Furthermore, $g : E_{\xi, \rho} \to
E_{g \xi, \rho'}$ is an isomorphism. Now, by the $\Gamma$-invariance
of $\sigma := \sigma_m(P)$, we have
$(g^{-1}\sigma)(\xi)=g^{-1}(\sigma(g\xi))g=\sigma(\xi)$ therefore
$\pi_{\xi,\rho'}(\sigma)$ is invertible if, and only if,
$\pi_{g\xi,\rho}(\sigma)$ is.

For the equivalence of (1) and (3), we can assume that $m=0$.  Recall
first that the density of $\Xi_0$ in $\Xi$ established in Theorem
\ref{theorem.non.principal} gives that the family of representations
\begin{equation*}
   \maF_{0} \ede \{\pi_{\xi,\rho} \, \mid\ (\xi,\rho) \in
   X^\alpha_{M,\Gamma}, \, \Gamma_\xi = \Gamma_0 \}
\end{equation*}
is faithful for the $C^*$-algebra $A_M^\Gamma/\ker (\maR_M)$ (see
e.g.\ \cite[Theorem 5.1]{Roc03}). In other words, the restriction
morphism
\begin{equation*}
  A_M^\Gamma/\ker (\maR_M) \to \bigoplus_{\substack{\rho \in
      \widehat{\Gamma}_0, \\ \rho \subset \alpha|_{\Gamma_0}}}
  \maC\big((T^*M \smallsetminus \{ 0\} )_{\Gamma_0}, \End
  (E_\rho)^{\Gamma_0}\big)
\end{equation*}
is injective. Since $(T^*M)_{\Gamma_0}$ is dense in $T^*M^{\Gamma_0}$,
it follows that the restriction morphism
\begin{equation*}
   R_M : A_M^\Gamma/\ker (\maR_M) \to \bigoplus_{\substack{\rho \in
       \widehat{\Gamma}_0, \\ \rho \subset \alpha|_{\Gamma_0}}}
   \maC\big(T^*M^{\Gamma_0} \smallsetminus \{ 0\} , \End
   (E_\rho)^{\Gamma_0}\big)
\end{equation*}
is also injective.

Let us write $\alpha_{|\Gamma_0} = \bigoplus_{\rho \in
  \widehat{\Gamma}_0} m_\rho \rho$, with multiplicities $m_\rho \ge
0$. {By considering the representations $\rho$ with $m_\rho > 0$, we
  see that} there is an injective vector bundle morphism over the
manifold $M^{\Gamma_0}$ defined by
\begin{equation}
  \label{eq.isomorphism_Etimesalpha}
  \Psi: \, \bigoplus_{\substack{\rho \in \widehat{\Gamma}_0, \\ \rho
      \subset \alpha|_{\Gamma_0}}} \End(E_\rho)^{\Gamma_0}
  \simeq \bigoplus_{\substack{\rho \in \widehat{\Gamma}_0, \\ \rho
      \subset \alpha|_{\Gamma_0}}} \End((E\otimes\rho)^{\Gamma_0})
  \hookrightarrow \End((E\otimes\alpha)^{\Gamma_0}) \, ,
\end{equation}
where the last morphism maps any element $T
\in \End((E\otimes\rho)^{\Gamma_0})$ to a direct sum of copies of $T$
acting on the direct summand
$\big[(E\otimes\rho)^{\Gamma_0}\big]^{m_\rho} \subset (E \otimes
\alpha)^{\Gamma_0}$.

Condition (3) amounts to the fact that
\begin{equation*}
   \Psi(R_M(\sigma^\Gamma_0(P))) \in \maC^\infty (T^*M^{\Gamma_0}
   \smallsetminus \{ 0\} ; \End((E\otimes \alpha)^{\Gamma_0})
\end{equation*}
is invertible. To establish that (1) $\Leftrightarrow$ (3), we thus
need to prove that $P$ is $\alpha$-elliptic if, and only if,
$\Psi(R_M(\sigma^\Gamma_0(P)))$ is invertible.

Recall the definition of the symbol algebras $A_M$ from Equation
\eqref{eq.def.ABZ}. We have that $P \in \clop$ is $\alpha$-elliptic
if, and only if, the image of $\sigma_0^\Gamma(P)$ in the quotient
algebra $A_M^\Gamma/\ker (\maR_M)$ is invertible (by the determination
of $\ker(\maR_M)$ in Remark \ref{rem.Xi=X_M_Gamma} or Proposition
\ref{prop.Xi=X_M_Gamma}). But since both $\Psi$ and $R_M$ are
injective, $\Psi \circ R_M$ is injective on $A_M^\Gamma/\ker
(\maR_M)$. Thus $\sigma_0^\Gamma(P)$ is invertible in the quotient
algebra $A_M^\Gamma/\ker (\maR_M)$ if, and only if,
$\Psi(R_M(\sigma^\Gamma_0(P)))$ is invertible. As we have seen above,
this amounts to (1) $\Leftrightarrow$ (3).
\end{proof}

\subsection{Fredholm conditions and Hodge and index theory}
We continue to assume that $M$ is a compact smooth manifold. We
have the following $\Gamma$--equivariant version of Atkinson's
theorem. (Recall that $\maK(V)$ denotes the algebra of compact
  operators acting on the Hibert space $V$.)

\begin{proposition} \label{prop.Atkinson}
Let  $V$ be a Hilbert
  space with a unitary action of $\Gamma$ and $P\in \maL(V)^\Gamma$
be a $\Gamma$--equivariant bounded operator on $V$. We have that $P$
is Fredholm if, and only if, it is invertible modulo $\maK(V)^\Gamma$,
in which case, we can choose the parametrix (i.e.\ the inverse modulo
the compacts) to also be $\Gamma$-invariant.
\end{proposition}

\begin{proof}
See for example \cite[Proposition 5.1]{BCLN1}.
\end{proof}

Since $\pi_\alpha(\maK(L^2(M; E))^\Gamma) = \maK(L^2(M;
E)_\alpha)^{\Gamma}$, we obtain the following corollary.

\begin{corollary} \label{cor.Atkinson}
Let $P \in \clop^\Gamma$ and $\alpha \in \widehat{\Gamma}$.  We have
that $\pi_\alpha(P)$ is Fredholm on $L^2(M; E)_\alpha$ if, and only
if, $\pi_\alpha(P)$ is invertible modulo $\pi_\alpha(\maK(L^2(M;
E))^\Gamma)$ in $\pi_\alpha(\clop^\Gamma)$.
\end{corollary}

We are now in a position to prove the main result of this paper,
Theorem \ref{thm.main1}.

\begin{proof}[Proof ot Theorem \ref{thm.main1}]
As in \cite[Section 2.6]{BCLN1}, we may assume that $P \in
\clop^\Gamma$. Corollary \ref{cor.Atkinson} then states that
$\pi_\alpha(P)$ is Fredholm if, and only if, the image of its symbol
$\sigma(P)$ is invertible in the quotient algebra
\begin{equation*}
    \maR_M(A_M^\Gamma) \seq \pi_\alpha(\clop^\Gamma)/
    \pi_\alpha(\maK(L^2(M; E) )^\Gamma).
\end{equation*}
According to Proposition \ref{prop.computation.Xi} and Remark
\ref{rem.Xi=X_M_Gamma} following it, the primitive spectrum $\Xi$ of
$\maR_M(A_M^\Gamma)$ identifies with $X_{M,\Gamma}^\alpha$.  Therefore
$\maR_M(\sigma(P))$ is invertible if, and only if, the endomorphism
$\pi_{\xi,\rho}(\sigma(P))$ is invertible for all $(\xi,\rho) \in
X_{M,\Gamma}^\alpha$, i.e.\ if, and only if, $P$ is $\alpha$-elliptic
(see Definition \ref{def.chi.ps}).
\end{proof}

\begin{remark}
Let $P : H^s(M; E) \to H^{s-m}(M; E)$ be an order $m$, classical
pseudodifferential operator. Since the index of Fredholm operators is
invariant under small perturbations and under compact perturbations,
we obtain that the index of $\pi_\alpha(P)$ depends only on the
homotopy class of its $\alpha$-principal symbol $\sigma_m^\alpha(P)$.
\end{remark}

An alternative approach to the Fredholm property (Theorem
\ref{thm.main1}) can be obtained from the following theorem. Recall
that $X^{\alpha}_{M, \Gamma}$ was defined in
\eqref{eq.def.Xalpha}. Below, by $\pa$ we shall denote the connecting
morphism in the six-term $K$-theory exact sequence associated to
  a short exact sequence of $C^*$-algebras. Recall that
$\sigma_0^\alpha$ is the $\alpha$-principal symbol map, see Definition
\ref{def.chi.ps}.

\begin{theorem}\label{thm.index}
Let $X^\alpha_{M,\Gamma}(E) := X^\alpha_{M,\Gamma}\cap \Omega_M$ and
let us denote by $\maC(X^\alpha_{M, \Gamma}(E)/\Gamma)$ the algebra of
restrictions of $A_M^\Gamma := \maC(S^*M; \End(E))^\Gamma$ to
$X^\alpha_{M, \Gamma}(E)/\Gamma$. Using the notation of Corollary
\ref{cor.Atkinson}, we have an exact sequence
\begin{equation*}
   0 \to \maK \to \pi_{\alpha}(\clop^\Gamma)
   \stackrel{\sigma_0^\alpha}{\longrightarrow} \maC(X^\alpha_{M,
     \Gamma}(E)/\Gamma) \to 0 \,.
\end{equation*}
Let $\pa : K_1(\maC(X^\alpha_{M, \Gamma}(E)/\Gamma)) \to \ZZ \simeq
K_0(\maK)$ be the associated connecting morphism and let $P \in
\clop^\Gamma$ be such that $\pi_{\alpha}(P)$ is Fredholm. Then
\begin{equation*}
  \ind(\pi_\alpha(P)) \seq \dim (\alpha) \pa [\sigma_0^\alpha(P)]\,.
\end{equation*}
\end{theorem}

\begin{proof}
The exactness of the sequence follows from the proof of Corollary
\ref{cor.Atkinson} and the fact that $\maK(L^2(M; E)_\alpha)^\Gamma
\simeq \maK$, the algebra of compact operators on a model separable
Hilbert space $\maH$. Under this isomorphism, the resulting
representation of $\maK$ on $L^2(M; E)_\alpha$ is isomorphic to
$\dim(\alpha)$ times the standard representation of $\maK$ on $\maH$.
This justifies the factor $\dim(\alpha)$. The rest follows from the
fact that the index is the connecting morphism in $K$-theory for the
Calkin exact sequence. See \cite{HigherDocumenta} for more details.
\end{proof}

\begin{remark}
As in \cite{HigherDocumenta}, it follows that the index of
$\pi_{\alpha}(P)$ with $P \in \psi^0(M;E)^\Gamma$ is the pairing
between a cyclic cocycle $\phi$ on $\maC^\infty(X_{M, \Gamma})$ (the
algebra of principal symbols of operators in $\psi^0(M;E)^\Gamma$) and
the $K$-theory class of the $\alpha$-principal symbol of $P$
\cite{ConnesBook}. See also \cite{gayral, KaroubiBook, LodayQuillen,
  LeschMosPflaum, ManinBook, Tsygan}. Lemma \ref{lemma.local2} gives that
the restriction of this cyclic cocycle to the principal orbit bundle
is the usual Atiyah-Singer cocycle (i.e.\ the cocycle that yields the
Atiyah-Singer index theorem in cyclic homology \cite{CMLocal, LMP,
  HigherDocumenta, Perrot}, which thus corresponds, after suitable
rescaling, to the Todd class). The full determination of the class of
the index cyclic cocycle $\phi$ require, however, a non-trivial use of
cyclic homology, since the quotient algebra $\maC^\infty(X_{M,
  \Gamma})$ is non-commutative, in general.
\end{remark}

\begin{remark}\label{rem.Hodge}
As for the case of compact complex varieties \cite{GriffithHarrisBook,
  WellsBook}, we can consider complexes of operators \cite{BL92} and
the corresponding notion of $\alpha$-ellipticity. In particular, we
obtain the finiteness of the corresponding cohomology groups if the
complex is $\alpha$-elliptic. This is related to the Hodge theory of
singular spaces \cite{AlbinHodge, MazzeoHodge, Brion, BL93,
  CheegerHodge, Teleman80}.
\end{remark}

\subsection{Special cases}
We now specialize our main result to some particular cases.

\subsubsection{The abelian group case \cite{BCLN1}}
Many statements and definitions become easier in the case of abelian
groups. In particular, if $\Gamma_i$, $i=1,2$, are both abelian, then
the irreducible representations $\alpha_i \in \widehat{\Gamma}_i$ are
characters, that is, morphisms $\alpha_i : \Gamma_i \to \CC^*$, and we
have that they are $H$-associated for some subgroup $H$ if, and only
if, $\alpha_1\vert_{H} = \alpha_2\vert_{H}$.

Let $\alpha$ be an irreducible representation of $\Gamma$. When
$\Gamma$ is abelian, the conjugacy class of isotropy subgroups
corresponding to the principal orbit type of the action has only one
element, namely $\Gamma_0$.  In that case, the set
$X_{M,\Gamma}^\alpha$ defined in Equation \eqref{eq.def.Xalpha} of the
introduction has the simpler expression:
\begin{equation*}
  X_{M,\Gamma}^\alpha = \{(\xi, \rho) \mid \, \xi \in
    T^*M\smallsetminus \{0\},\, \rho \in \widehat{\Gamma}_\xi,
    \rho\vert_{\Gamma_0} = \alpha\vert_{\Gamma_0}\, \}.
\end{equation*}

As a consequence, it is easier to check the $\alpha$-ellipticity for
an operator $P$ in the abelian case. Let $E,F$ be $\Gamma$-equivariant
vector bundles over $M$ and set $\alpha_0 :=
\alpha\vert_{\Gamma_0}$. We then recover the main result of
\cite{BCLN1}. Indeed, Theorem \ref{thm.main1} can then be stated as
follows:

\begin{theorem}\cite[Theorem 1.2]{BCLN1} \label{thm.abelian_case}
Let $\Gamma$ be a finite, \emph{abelian} group acting on a smooth,
compact manifold $M$ and let $P \in \psi^m(M; E, F)^\Gamma$.  Then,
for any $s \in \RR$, the following are equivalent:
\begin{enumerate}
  \item the operator $\pi_\alpha(P) : H^s(M; E)_\alpha \, \to \,
    H^{s-m}(M; F)_\alpha$ is Fredholm,
  \item for all $(x,\xi) \in T^*M\smallsetminus\{0\}$, $\rho \in
    \widehat{\Gamma}_\xi$ such that $\rho_{|\Gamma_0}=\alpha_0$, the
    restriction of $\sigma(P)(x,\xi)$ defines an isomorphism
\begin{equation*}
   \pi_{\rho}(\sigma(P)(x,\xi)) : E_{x\rho} \to F_{x\rho}\,.
\end{equation*}
\end{enumerate} 
\end{theorem}

\subsubsection{Scalar operators} 
Our main theorem becomes quite explicit when we are dealing with
scalar operators, i.e.\ when the vector bundles $E_i= M\times
\mathbb{C}$, where $\mathbb{C}$ denotes the trivial representation of
$\Gamma$.

\begin{proposition}
Let $P : H^s(M)\rightarrow H^{s-m}(M)$ be a $\Gamma$-invariant
pseudodifferential operator. Let $\alpha\in \widehat{\Gamma}$. Then
$P$ is $\alpha$-elliptic if, and only if, $\sigma(P)(\xi)$ is
invertible for all $\xi\in T^*M\smallsetminus \{0\}$ such that $\alpha$ is
$\Gamma_0$-associated to the trivial (constant 1) representation of
$\Gamma_\xi$. 
\end{proposition}

\begin{proof}
Let $\widehat{1}_{\Gamma_\xi}$ denote the trivial representation of
$\Gamma_{\xi}$ and let $(\xi,\rho)\in X_{M,\Gamma}^\alpha$. If $\rho
\neq \widehat{1}_{\Gamma_\xi}$ then $\mathbb{C}_\rho=0$ and then
$\pi_\rho(\sigma(P)(\xi)) : 0 \rightarrow 0$ is invertible. Now if
$\rho = \widehat{1}_{\Gamma_\xi}$ then $(\xi,\rho)\in
X_{M,\Gamma}^\alpha$ if, and only if, $\alpha$ is
$\Gamma_0$-associated to $\widehat{1}_{\Gamma_\xi}$.
\end{proof}

\subsubsection{Trivial actions}
Assume that $\Gamma$ acts trivially on $M$
(in particular, $M$ is then also connected). Our assumption
implies that $\Gamma_0 = \Gamma_\xi=\Gamma$, for all $\xi \in
T^*M\smallsetminus \{0\}$. It follows that $\rho \in \widehat{\Gamma}_\xi$
is $\Gamma_0$-associated to $\alpha \in \widehat{\Gamma}$ if, and only
if, $\alpha= \rho$.

Let $E \to M$ be a $\Gamma$-equivariant vector bundle. For any $x \in
M$, recall that we denote $E_{x\alpha}$ the $\alpha$-isotypical
component of $E_x$. Assuming $M$ to be connected, we have that
$E_\alpha = \bigcup_{x \in M}E_{x\alpha}$ is a $\Gamma$-equivariant
sub-vector bundle of $E$. Our main result then becomes the
following statement.

\begin{proposition}
Assume that $\Gamma$ acts trivially on $M$ and let $\alpha \in
\widehat{\Gamma}$. Let $E,F$ be two $\Gamma$-equivariant vector
bundles over $M$ and let $P \in \psi^m(M;E,F)^\Gamma$.  Then for any
$s \in \RR$, the following are equivalent
\begin{enumerate}
    \item \label{item.Fredholm.operator} $\pi_{\alpha}(P) :
      H^s(M;E_{\alpha}) \to H^{s-m}(M;F_{\alpha})$ is Fredholm,
\item \label{item.symbol.restriction} for all $(x,\xi) \in
  T^*M\smallsetminus\{0\}$, the morphism
  \begin{equation*}
    \pi_\alpha(\sigma(P)(x,\xi)) : E_{x\alpha} \to F_{x\alpha}
  \end{equation*}
  is invertible,
\item \label{item.symbol.tensor} for all $(x,\xi) \in
  T^*M\smallsetminus\{0\}$, the morphism
  \begin{equation*}
\sigma_m(P) \otimes \mathrm{id}_{\alpha^*}(x,\xi) :
\mathrm{Hom}_\Gamma(\alpha ,E_{x}) \rightarrow
\mathrm{Hom}_\Gamma(\alpha ,F_{x})
  \end{equation*}
is invertible.
\end{enumerate}
\end{proposition}

Of course, the above result is nothing but the classical condition
that the elliptic operator $p_{F_\alpha} P p_{E_\alpha} \in
\psi^m(M;E_\alpha,F_\alpha)$ be Fredholm.

\begin{proof}
The equivalence between \eqref{item.Fredholm.operator} and
\eqref{item.symbol.restriction} is a direct consequence of Theorem
\ref{thm.main1}. Let us check the equivalence of
\eqref{item.Fredholm.operator} and \eqref{item.symbol.tensor}.  First
note that
\begin{equation*}
    (H^s(M,E)\otimes \alpha)^\Gamma = H^s(M,(E\otimes \alpha)^\Gamma),
\end{equation*} 
since the action of $\Gamma $ on $M$ is trivial.  The operator
$\pi_\alpha(P)$ is Fredholm if, and only if, the pseudodifferential
operator $P_\alpha : H^s(M,\mathrm{Hom}(\alpha,E)^\Gamma ) \rightarrow
H^{s-m}(M, \mathrm{Hom}(\alpha,F)^\Gamma)$ defined for any $v^* \in
\alpha^*$ and $s \in \maC^{\infty}(M,E)$ by $P_\alpha(v^*s)=v^*Ps$ is
Fredholm. Furthermore, the operator $P_\alpha$ is Fredholm if, and
only if, it is elliptic, that is if, and only if, $\sigma_m(P) \otimes
\mathrm{id}_{\alpha^*}(x,\xi) : \mathrm{Hom}_\Gamma(\alpha ,E_{x})
\rightarrow \mathrm{Hom}_\Gamma(\alpha ,F_{x})$ is invertible for any
$(x,\xi)\in T^*M\smallsetminus\{0\}$. Note that the invertibility of
$\sigma_m(P) \otimes \mathrm{id}_{\alpha^*}(x,\xi)$ is equivalent to
the invertibility of $\pi_\alpha(\sigma_m(P)(x,\xi))$ by definition,
which is consistent with~\eqref{item.symbol.restriction}.
\end{proof}

\subsubsection{Free action on a dense subset}

As in the previous sections, the group $\Gamma$ is finite and acts
continuously on the manifold $M$. We consider vector bundles $E,F \to
M$.

We have following corollary of the last few results in Section
\ref{sec.principal.symbol}.
\begin{corollary} \label{cor.all}
Let us assume that $\Gamma$ acts freely on a dense open subset of
$M$. Then $\Xi = \Prim (A_M^\Gamma)$.
\end{corollary}
\begin{proof}
The assumption on the action implies that $\Gamma_0 = \{1\}$. If $\xi
\in T^*M\smallsetminus\{0\}$ and $\rho \in \widehat{\Gamma}_\xi$, then
$\rho$ and $\alpha$ are always $\{1\}$-associated. The Corollary then
follows from Proposition \ref{prop.computation.Xi}.
\end{proof}
Similarly, we have the following result.

\begin{proposition} \label{prop.free.dense}
Assume that $\Gamma$ acts freely on a dense subset in $M$, and let $P
\in \psi^m(M;E,F)^\Gamma$. For any $\alpha \in \widehat{\Gamma}$, we
have that $P$ is $\alpha$-elliptic if, and only if, $P$ is elliptic.
\end{proposition}

\begin{proof}
It follows from Corollary \ref{cor.all} that
$X_{M,\Gamma}^\alpha=X_{M,\Gamma}$. Thus the operator $P_\alpha$ is
$\alpha$-elliptic if, and only if, the sum $\bigoplus_{\rho \in
  \widehat{\Gamma}_\xi} \pi_{\rho}(\sigma_m(P)(\xi))=\sigma_m(P)(\xi)$
is invertible for all $\xi \in T^*M\smallsetminus\{0\}$, that is, if, and
only if, $P$ is elliptic.
\end{proof}

\subsection{Simonenko's localization principle}
In this section, we obtain an equivariant version of Simonenko's
  principle \cite{simonenko1965new}. In this subsection and the rest
  of the paper, we consider a compact Lie group $G$ instead of
  $\Gamma$.

\subsubsection{Simonenko's general principle} \label{ssub.Simonenko}
Let $A$ be a unital $C^*$-algebra and $Z \simeq \maC(\Omega_Z)$ a
unital sub-$C^*$-algebra in $A$, i.e.\ $1_Z=1_A$.  An element $a \in
A$ is said to {\em have the strong Simonenko local property} with
respect to $Z$ if, for every $\phi, \psi \in Z$ with compact disjoint
supports, $\phi a \psi = 0$. 

\begin{lemma}\label{lemma:commute} 
The set $B \subset A$ of elements $a$ satisfying the strong Simonenko
local property is the set of elements of $A$ commuting with $Z$.
\end{lemma}

\begin{proof}
We are going to show that the set of elements $a \in A$ with the
strong Simonenko local type property is a $C^*$-algebra $B$ containing
$Z$ and that every irreducible representation of $B$ restricts to a
scalar valued representation on $Z$, and hence that $Z$
commutes with $B$.

Let us show first that $B$ is a sub-$C^*$-algebra of $A$. Note that
$B$ is not empty since $Z \subset B$. To show that $B$ is a
sub-$C^*$-algebra, the only fact that is non-trivial to prove is that
$ab \in B$, for all $a,b \in B$. Let $\phi $ and $\psi \in Z $ with
disjoint compact supports and let $\theta$ be a function equal to $1$
on $\mathrm{supp}(\psi)$ and $0$ on $\mathrm{supp}(\phi)$, which
exists by Urysohn's lemma. Then we have
\begin{equation}
  \phi ab\psi = \phi a(\theta +1-\theta)b \psi =\phi a \theta b\psi
  +\phi a(1-\theta )b\psi = 0,
\end{equation}
since $\phi a \theta=0$ and $(1-\theta )b\psi=0$.

Let $\pi : B \rightarrow \mathcal{L}(H)$ be an irreducible
representation of $B$.  First, let us show that for any $\phi, \psi
\in Z$ with disjoint support, we either have $\pi(\phi)=0$ or
$\pi(\psi) =0$. Indeed we have $\pi(\phi)\pi(a)\pi(\psi)=0$ since
$\phi a \psi =0$, for any $a \in B$. Assume that $\pi(\psi) \neq 0$
then there is $\eta \in H$ such that $\pi(\psi)\eta \neq 0$. Now,
$\pi$ is irreducible so we get that the set $\{\pi(a)\pi(\psi)\eta
,\ a \in B \}$ is dense in $H$. Thus $\pi(\phi)=0$ on a dense subspace
of $H$ and so on $H$.

Assume now that $\pi(Z) \neq \CC 1_H$. Then there exist two distinct
characters $\chi_0, \chi_1 \in \Sp(\pi(Z))$. Denote by $h_\pi :
\Sp(\pi(Z)) \to \Sp(Z)$ the injective map adjoint to $\pi$, and choose
$\phi, \psi \in \maC(\Sp(Z))$ with disjoint supports such that
$\phi(h_\pi(\chi_0)) = 1$ and $\psi(h_\pi(\chi_1)) = 1$. Then
$\pi(\phi)(\chi_0) = 1$ and $\pi(\psi)(\chi_1) = 1$, which contradicts
the fact that either $\pi(\phi) = 0$ or $\pi(\psi) = 0$.
\end{proof}

Recall that a family $(\varphi_i)_{i\in I}$ of morphisms of a
$C^*$-algebra $A$ is said to be {\em exhaustive} if any primitive
ideal contains some $\ker (\varphi_i)$, for a suitable $i\in I$
\cite{nistor2014exhausting}.  Then Remark \ref{rem.central.char}
  gives that the family of morphisms
\begin{equation}\label{eq.central.exhaustive}
   \chi_\omega : A \to A/\omega A,
\end{equation}
for $\omega \in \Omega_Z$, is exhaustive for $A$.

\begin{definition}\label{def.loc.inv}
Denote by $\maH =L^2(M)$. An operator $P \in \maL(\maH)$ is said to be
{\em locally invertible} at $x\in M$ if there are a neighbourhood
$V_x$ of $x$ and operators $Q_1^x$ and $Q_2^x \in \maL(\maH)$ such
that
\begin{equation}
  Q_1^xP\phi \seq \phi\quad \mathrm{and} \quad \phi PQ_2^x \seq
  \phi,\qquad \mathrm{for\ any}\ \phi \in \maC_c(V_x).
\end{equation}
The operator $P$ is said to be {\em locally invertible} if it is
locally invertible at any $x\in M$.
\end{definition}

Let $\Psi_M \subset \maL(\maH)$ be the $C^*$-algebra of all $P \in
\maL(\maH)$ such that $\phi P \psi \in \maK(\maH)$, for all $\phi,
\psi \in \maC(M)$ with disjoint support. We denote by $\maB_M$ the
image of $\Psi_M$ in the Calkin algebra $\maQ(\maH) :=
\maL(\maH)/\maK(\maH)$. We know by Lemma \ref{lemma:commute} that
\begin{equation*}
  \maB_M=\{P \in \mathcal{Q}(H) \mid \text{$\phi P =P\phi$ for all
  $\phi \in \maC(M)$}\}.
\end{equation*}
Simonenko's principle is then \cite{simonenko1965new}:

\begin{proposition}[Simonenko's principle] \label{prop.simonenko}
If $P \in \Psi_M$, then $P$ is locally invertible if, and only if, it
is Fredholm.
\end{proposition}

We shall prove, in fact, a stronger version of this result, see
Proposition \ref{prop.simonenko.equivariant}.

\subsection{Compact (non-finite) groups}
We now
  allow for compact groups and try to see to what extent our results
  remain valid. To that end, we turn to an analog of Simonenko's
  principle. Let then $G$ be a compact Lie group acting smoothly on
  $M$. We continue to study Fredholm conditions for $\pi_\alpha(P)$,
  $\alpha \in \widehat{G}$.  Denote by $\mathcal{H}:=L^{2}(M,E)$ and
by $\maH_\alpha$ the $\alpha$-isotypical component associated to
$\alpha \in \widehat{G}$.

\begin{definition}\label{def.alpha.invertible}
We shall say that $P \in \maL(\maH)$ is {\em locally $\alpha$-invertible at
  $x\in M$} if there are a $G$-invariant neighbourhood $V_x$ of
$\Gamma x$ and operators $Q_1^x$ and $Q_2^x \in \maL(\maH_\alpha)$
such that
\begin{equation}\label{eq.equiv.local}
   Q_1^x\pi_\alpha(P)\phi \seq \phi\qquad \mbox{and} \qquad\phi
   \pi_\alpha(P) Q_2^x=\phi,
\end{equation}
as operators on $\maH_\alpha$, for any $\phi \in \maC(M)^G$
supported in $V_x$.
\end{definition}

We denote by $\Psi_M^G$ the $G$-invariant elements in the
$C^*$-algebra $\Psi_M$, which was defined in the previous subsection.

\begin{proposition}[Simonenko's equivariant principle]
\label{prop.simonenko.equivariant}
Let $P \in \Psi_M^G$. Then $P$ is locally $\alpha$-invertible if, and
only if, $\pi_\alpha(P)$ is Fredholm.
\end{proposition}

\begin{proof}
We now use the results of the last section for $Z = \maC(M)^G =
  \maC(M/G)$. Let $\maB_M^\alpha$ be the image of $\Psi_M^G$ in the
Calkin algebra $\maQ(\maH_\alpha)$. We know from Lemma
\ref{lemma:commute} that
  \begin{equation*}
    \mathcal{B}_{M}^\alpha=\{P \in \mathcal{Q}(\maH_\alpha) \mid \phi
    P =P\phi,\ \forall \phi \in \maC(M)^G\}.
  \end{equation*}

Assume that $P$ is locally $\alpha$-invertible, i.e.\ $\forall x \in
M$, there are a neighborhood $V_x$ of $G x$ and operators $Q^x_1,
Q^x_2 \in \maL(\maH_\alpha)$ such that $Q^x_1\pi_\alpha(P)\phi = \phi$
and $\phi \pi_\alpha(P)Q_2^x=\phi$, for any $\phi \in \maC(M)^G$
supported in $V_x$. Let $\chi_x$, be the family of
  representations of $\maB_G^\alpha$ introduced in Equation
  \eqref{eq.central.exhaustive}. We use the same notation for
$\pi_\alpha(P)$ and its image in $\maQ(\maH_\alpha)$. We have
that
\begin{equation*}
    \chi_x(Q_1^x\pi_\alpha(P)
    \phi)=\chi_x(Q_1^x)\chi_x(\pi_\alpha(P))\chi_x(\phi)=\chi_x(\phi).
\end{equation*}
Since $\chi_x(\phi) = 1$, we get:
\begin{equation*}
 \chi_x(Q_1^x)\chi_x(\pi_\alpha(P))=1. 
\end{equation*}
And similarly, 
\begin{equation*}
   \chi_x(\pi_\alpha(P)) \chi_x(Q_2^x)=1.
\end{equation*}
Therefore, $\chi_x(\pi_\alpha(P))$ is invertible for all
  $x$. Since the family $\chi_x$ is exhaustive, it follows that that
$\pi_\alpha(P)$ is invertible in $\mathcal{B}_{M}^\alpha$ and so it is
Fredholm.

Now assume that $\pi_\alpha(P)$ is Fredholm and let $Q$ be an inverse
modulo $\mathcal{K}(\maH_\alpha)$ for $\pi_\alpha(P)$,
i.e.\ $\pi_\alpha(P)Q=id+K$ and $Q\pi_\alpha(P)=id+K^\prime$, with
$K,K' \in \maK(\maH_\alpha)$. Using Proposition
  \ref{prop.Atkinson} and Lemma \ref{prop.image}, we can assume that
  $K=\pi_\alpha(k)$ and $K^\prime=\pi_\alpha(k^\prime)\in
  \maK(\maH_\alpha) = \pi_\alpha(\mathcal{K}(\maH)^G)$, with $k,
    k' \in \mathcal{K}(\maH)^G$. Let $\chi \in \maC(M)^G$ be equal
to $1$ on a $G$-invariant neighbourhood $V_x$ of $G x$ and let $\phi
\in \maC(M)^G$ be supported in $V_x$ then
$$\phi \chi \pi_\alpha(P)Q\chi =\phi \chi^2+\phi \chi K\chi \qquad
\mathrm{and}\qquad \chi \pi_\alpha(P)Q\chi \phi = \chi^2\phi + \chi
K^\prime\chi \phi.$$ Since $\phi$ is supported in $V_x$, we have
$\phi\chi=\phi$ and so
$$\phi \pi_\alpha(P)Q\chi =\phi (1+ \chi K\chi)\qquad
\mathrm{and}\qquad \pi_\alpha(P)Q\chi \phi = (1+ \chi
K^\prime\chi)\phi.$$ 

As $V_x$ becomes smaller and smaller, we have that $\chi$
  converges strongly to $0$. Since $K$ is compact, we obtain that
  $\|\chi K \chi \| \to 0$. Thus, by choosing $V_x$ small enough, we
  may assume that $\|\chi K\chi\| < 1$ and $\|\chi K' \chi\| < 1$.

It follows that $(1+ \chi K\chi)$ and $(1+ \chi K^\prime\chi)$ are
invertible and this implies
\begin{equation*}
   \phi \pi_\alpha(P)\big(Q\chi(1+ \chi K\chi)^{-1}\big)
   =\phi\quad\ \mathrm{and}\quad \big((1+ \chi K^\prime\chi)^{-1}\chi
   Q\big)\pi_\alpha(P)\phi =\phi,
\end{equation*}
i.e.\ $P$ is locally $\alpha$-invertible.
\end{proof}

\begin{corollary}
Assume that $M$ is compact, $\Gamma$ is a finite group and $M/\Gamma$
connected.  Let $P\in\psi(M;E,F)^\Gamma$ and $\alpha \in
\widehat{\Gamma}$. Then the following are equivalent:
\begin{enumerate}
  \item $\pi_\alpha(P) : H^s(M;E)_\alpha \to H^{s-m}(M;F)_\alpha$ is
    Fredholm for any $s \in \RR$,
  \item $P$ is $\alpha$-elliptic,
  \item $P$ is locally $\alpha$-invertible.
\end{enumerate}
\end{corollary}
\begin{proof}
The first equivalence is given by Theorem \ref{thm.main1}.  Now since
a finite group is compact Proposition \ref{prop.simonenko.equivariant}
implies that $(1)$ is equivalent to $(3)$.
 \end{proof}

\subsubsection{Transversally elliptic operators}
Assume that $M$ is a compact smooth manifold and that $G$ is a compact
Lie group acting on $M$. Denote by $\mathfrak{g}$ the Lie algebra of
$G$. Then any $X \in \mathfrak{g}$ defines as usual the vector field
$X^*_M$ given by $X^*_M(m)=\frac{d}{dt}_{|_{t=0}} e^{tX}\cdot
m$. Denote by $\pi : T^*M \rightarrow M$ the canonical projection and
let us introduce as in \cite{atiyahGelliptic} the $G$-transversal
space
\begin{equation*}
    T^*_G M \ede \{\alpha \in T^*M\ |\ \alpha(X^*_M(\pi(\alpha)))=0,
    \forall X\in \mathfrak{g}\}.
\end{equation*}
Recall that a $G$-invariant classical pseudodifferential operator $P$
of order $m$ is said {\em $G$-transversally elliptic} if its principal
symbol is invertible on $T^*_G M\smallsetminus \{0\}$
\cite{atiyahGelliptic, PVTrans}.

We may now state the classical result of
Atiyah and Singer \cite[Corollary 2.5]{atiyahGelliptic}. 
\begin{theorem}[Atiyah-Singer \cite{atiyahGelliptic}]
\label{Thm.atiyahGell}
Assume $P$ is $G$-transversally elliptic. Then, for every irreducible
representation $\alpha\in \widehat{G}$, $\pi_\alpha(P) : H^s(M;
E_0)_\alpha \, \to \, H^{s-m}(M; E_1)_\alpha,$ is Fredholm.
\end{theorem}
Note that this implies that Theorem \ref{thm.main1} is not true
anymore for if $G$ is non-discrete. In particular, we obtain the
following consequence of the localization principle.

\begin{corollary}
Assume that $M$ is compact and that $G$ is a compact Lie group and let
$P\in \op{m}^G$ be a $G$-transversally elliptic operator. Then $P$ is
locally $\alpha$-invertible for any $\alpha \in \widehat{G}$, as in
Definition \ref{def.alpha.invertible}.
\end{corollary}

\begin{proof}
Using Theorem \ref{Thm.atiyahGell} we obtain that $\pi_\alpha(P)$ is
Fredholm.  Therefore by Proposition \ref{prop.simonenko.equivariant}
$P$ is $\alpha$-invertible.
\end{proof}



\setlength{\baselineskip}{4.75mm}
\bibliographystyle{plain}

\bibliography{strEqCalc}

\end{document}